\documentclass[11pt]{article}
\usepackage{amssymb}
\usepackage{amsbsy}
\usepackage[latin1]{inputenc}
\usepackage{amsthm}
\usepackage[dvips]{graphicx}
\usepackage{graphicx} 
\usepackage{subfigure}
\usepackage{pst-eucl}
\usepackage[latin1]{inputenc}
\usepackage[english]{babel}
\usepackage{amsmath,amssymb,graphics,mathrsfs}
\usepackage{amsmath,amssymb,latexsym}
\usepackage{graphicx,color}
\usepackage[T1]{fontenc}
\usepackage[active]{srcltx}
\usepackage{multicol}
\usepackage[latin1]{inputenc}
\usepackage{pst-all}
\usepackage{enumerate}
\usepackage{pstricks}
\usepackage{pstricks-add}
\usepackage{setspace}
\usepackage{soul}
\usepackage{cancel}
\usepackage{nonfloat}
\usepackage[margin=10pt,font=footnotesize,labelfont=bf,labelsep=endash]{caption}
\usepackage[left=4cm,top=3cm,right=2.4cm,bottom=3.2cm]{geometry}

\usepackage[colorlinks=true,citecolor=red,linkcolor=blue,urlcolor=RubineRed,pdfpagetransition=Blinds,pdftoolbar=false,pdfmenubar=false]{hyperref}

\newtheorem{definition}{Definition}
\newtheorem{theorem}{Theorem}
\newtheorem{corol}{Corollary}
\newtheorem{lemma}{Lemma}
\newtheorem{remark}{Remark}

\begin{document}
\title{A  regularity criterion for a 3D chemo-repulsion system and its application to  a bilinear optimal  control problem}
\author{F. Guill\'en-Gonz\'alez$^1$, E. Mallea-Zepeda$^2$, M.A. Rodr\'iguez-Bellido$^3$}
\date{\small$^{1,3}$\it Dpto. Ecuaciones Diferenciales y An\'alisis Num\'erico and IMUS\, Universidad de Sevilla,\\
Facultad de Matem\'aticas, C/ Tarfia, S/N, 41012, Spain\\
\small$^2$\it Departamento de Matem\'atica, Universidad de Tarapac\'a, Arica, Chile}
\maketitle
\footnotetext{$^1$ E-mail: {\tt guillen@us.es}}
\footnotetext{$^2$E-mail: {\tt emallea@uta.cl}}
\footnotetext{$^3$E-mail: {\tt angeles@us.es}}
\date{}

\begin{abstract}
In this paper we study a bilinear optimal control problem associated to a 3D chemo-repulsion model with linear production.
We prove the existence of weak solutions and we establish a regularity  criterion to get global in time strong solutions.
As a consequence, we deduce the existence of a global optimal solution with bilinear control and, using a Lagrange multipliers theorem, we derive first-order optimality conditions for local optimal solutions.



\end{abstract}

\section{Introduction}
\label{intro}

The chemotaxis phenomenon is understood as the directed movement of live organisms in response to chemical gradients. Keller and Segel
\cite{keller-segel} proposed a mathematical model that describes chemotactic aggregation of cellular slime molds which move preferentially towards
relatively high concentrations of a chemical substance secreted by the amoebae themselves, which is called {\it chemo-attraction} with production. When the regions
of high chemical concentration generate a repulsive effect on the organisms, the phenomenon is called {\it chemo-repulsion}.

In this work we study an optimal  control problem subject to a chemo-repulsion  with linear production system in which a bilinear control acts  injecting or extracting 
chemical substance on a subdomain of control $\Omega_c\subset\Omega$. Specifically, we consider $\Omega\subset\mathbb{R}^3$ be a simply connected
bounded domain  with boundary $\partial\Omega$ of class $C^2$  and $(0,T)$ a time interval, with $0<T<+\infty$. Then we study a control problem related to
the following system in the  time-space domain $Q:=(0,T)\times\Omega$,
\begin{equation}\label{eq1}
\left\{\begin{array}{rcl}
\partial_tu-\Delta u&=&\nabla\cdot(u\nabla v),\\
\partial_tv-\Delta v+v&=&u+f\,v\,\chi_{_{\Omega_c}},\\
\end{array}
\right.
\end{equation}
with initial conditions 
\begin{equation}\label{eq2}
u(0,\cdot)=u_0\ge0,\ v(0,\cdot)=v_0\ge0\ \mbox{ in }\Omega,
\end{equation}
and non-flux boundary conditions
\begin{equation}\label{eq3}
\dfrac{\partial u}{\partial{\bf n}}=0,\quad \dfrac{\partial v}{\partial{\bf n}}=0\ \mbox{ on }(0,T)\times\partial\Omega,
\end{equation}
where ${\bf n}$ denotes the outward unit normal vector to $\partial\Omega$. In (\ref{eq1}), the unknowns are the cell density $u(t,x)\ge 0$ and chemical concentration $v(t,x)\ge 0$. The function $f=f(t,x)$ 
denotes a bilinear control  acting in the chemical equation. 
We observe that in the subdomains of $\Omega$ where
$f\ge0$ the chemical substance is injected, and conversely where $f\le0$ the chemical substance is extracted. 

System (\ref{eq1})-(\ref{eq3}) without control (i.e.~$f\equiv0$) has been studied in \cite{cieslak}, \cite{tao}. In \cite{cieslak}, the authors proved the global existence and uniqueness of smooth classical
solutions in 2D domains, and global existence of weak solutions in  dimension 3 and 4. In \cite{tao}, on a bounded convex domain $\Omega\subset\mathbb{R}^n$ ($n\ge 3$),  
it is proved that a modified system of (\ref{eq1})-(\ref{eq3}), changing the chemotactic term $\nabla\cdot(u\nabla v)$ by $\nabla\cdot(g(u)\nabla v)$ with an adequate density-dependent chemotactic function $g(u)$, has a unique global in time classical solution. This result is not applicable in our case, because $g(u)=u$ does not satisfies  the hypothesis imposed in \cite{tao}.

There is an extensive literature devoted to the study of control problems with PDEs, see for instance 
\cite{alekseev,casas-1,casas-2,karl-wachsmuth,kien-rosch,kunisch-trautmann,mallea-ortega-villamizar-1,mallea-ortega-villamizar-2,tachim,zhen}
and references therein. In all previous works, the control is of distributed or  boundary type. As far as know, the literature related to optimal control problems with PDEs and bilinear
control is scarce, see \cite{borzi,fister-mccarthy,guillen-mallea-rodriguez,kroner,vallejos}.

In the context of optimal control problems associated to chemotaxis models, the literature is also scarce. 
In \cite{fister-mccarthy,ryu} a 1D problem is studied. In \cite{fister-mccarthy} the authors analyzed 
two problems for a chemoattractant model.
The bilinear control  acts on the whole $\Omega$ in the cells equation. The existence of optimal control is proved and an optimality system is derived. 
Also, a numerical scheme for the optimality system is designed and some numerical simulations are presented. 
In \cite{ryu}  a boundary control problem for a chemotaxis reaction-diffusion system is studied. The control acts on the boundary  for the chemical substance, and the existence
of optimal solution is proved. A distributed optimal control problem for a two-dimensional model of cancer invasion has been studied in \cite{dearaujo}, 
proving the existence of optimal solution and deriving an optimality system.
Rodr\'iguez-Bellido et al.~\cite{rodriguez-rueda-villamizar} study a distributive optimal control problem related to a $3D$ stationary chemotaxis
model coupled with the Navier-Stokes equations ({\it chemotaxis-fluid system}). The authors prove the existence of an optimal solution and derive an
optimality system using a penalty method, taking into account that the relation control-state is multivalued. Ryu and Yagi \cite{ryu-yagi} study an extreme problem
for a chemoattractant $2D$ model, in which the control variable is distributed in the  chemical equation. They prove the existence of
optimal solutions, and  derive an optimality system, using the fact that the state is differentiable with respect to the control. 
Other studies related to controllability for the nonstationary Keller-Segel model and nonstationary chemotaxis-fluid system can be consulted in \cite{chaves-guerrero-1} and \cite{chaves-guerrero-2}, respectively.

In \cite{guillen-mallea-rodriguez},  an optimal bilinear control problem related to strong solutions of system (\ref{eq1})-(\ref{eq3}) in 2D domains was studied, 
proving  the existence and uniqueness of global strong solutions, and the existence of  global optimal control. Moreover, using a Lagrange multiplier theorem, first-order optimality conditions are derived.
Now, this paper can be seen as a 3D version of  \cite{guillen-mallea-rodriguez}.  In fact, similarly to \cite{guillen-mallea-rodriguez}, the main objective now 
is to prove the existence of global optimal solutions and to derive  optimality conditions, which 
will be  more complicated because the PDE system is considered in 3D domains. In this case, we distinguish two different types of solutions: {\it weak} and {\it strong}. The existence of weak solutions can be obtained under minimal assumptions (see Theorem \ref{weak_solution}).
However, such result is not sufficient to carry out the study of the control problem, due to the lack of regularity of  weak solutions.  In order to overcome this problem, 
we introduce a regularity criterion that allows to obtain a (unique)  strong solution of (\ref{eq1})-(\ref{eq3}) (see Theorem \ref{regularity}).
As far as we know, there are no results of global in time regularity of weak solutions of  system (\ref{eq1})-(\ref{eq3}) in $3D$ domains. 
This is similar to what happens with the Navier-Stokes equations (see \cite{temam}).

In this work, we deal with strong solutions of (\ref{eq1})-(\ref{eq3}) which allows us to analyze  the control problem.
However, we are going to  prove the existence of an optimal control  associated to strong solutions, assuming the existence of controls such that the associated strong solution exists.  
Following the ideas of \cite{casas-1,casas-2}, we  consider a regularity criterion in the objective functional such that any weak solution of
(\ref{eq1})-(\ref{eq3}) with this regularity is also a strong solution.

The paper is organized  as follow: In Section \ref{sec:2}, we fix the notation, introduce the functional spaces to be used
and we state a regularity result for linear parabolic-Neumann problems that will be used throughout this work. In Section \ref{sec:3} we give the definition of weak solutions of  (\ref{eq1})-(\ref{eq3}) and, by introducing a family of regularized
problems related to (\ref{eq1})-(\ref{eq3}) (its existence is deduced in the Appendix) and passing to the limit, prove the existence of weak solutions of system (\ref{eq1})-(\ref{eq3}). In Section \ref{sec:4} we give the definition of strong solutions of  (\ref{eq1})-(\ref{eq3}), and we establish a regularity  criterion under which  weak solutions of  (\ref{eq1})-(\ref{eq3}) are also strong solutions.
Section \ref{sec:5} is dedicated to the study of a bilinear control problem related to strong solutions of system (\ref{eq1})-(\ref{eq3}), proving the existence of an optimal solution and deriving
the first-order optimality conditions based on a Lagrange multipliers argument in Banach spaces. Finally, we obtain  a regularity result for these Lagrange multipliers.

\section{Preliminaries}
\label{sec:2}

We will introduce some notations. We will use the Lebesgue space
$L^p(\Omega)$, $1\le p\le+\infty$, with norm denoted by $\|\cdot\|_{L^p}$. In particular, the $L^2$-norm and
its inner product will denoted by $\|\cdot\|$ and $(\cdot,\cdot)$, respectively. We consider the usual Sobolev spaces
$W^{m,p}(\Omega)=\{u\in L^p(\Omega)\,:\, \|\partial^\alpha u\|_{L^p}<+\infty,\ \forall |\alpha|\le m\}$, 
with norm denoted by $\|\cdot\|_{W^{m,q}}$. When $p=2$, we write $H^m(\Omega):=W^{m,2}(\Omega)$ and we denote the respective norm by
$\|\cdot\|_{H^m}$. Also, we use the space $W_{\bf n}^{m,p}(\Omega)=\{u\in W^{m,p}(\Omega)\,:\, \frac{\partial u}{\partial{\bf n}}=0\mbox{ on }\partial\Omega\}$  ($m\ge2$) and its norm denoted
by $\|\cdot\|_{W^{m,p}_{\bf n}}$. 
If $X$ is a Banach space, we denote by $L^p(X)$ the space of valued functions in $X$ defined on the interval $[0,T]$ that are integrable in the Bochner sense, and its norm will be
denoted by  $\|\cdot\|_{L^p(X)}$. For  simplicity we denote $L^p(Q):=L^p(0,T;L^p)$ and its norm by $\|\cdot\|_{L^p(Q)}$. 
We also denote by $C([0,T];X)$ the space of continuous functions from $[0,T]$ into a Banach space $X$, whose norm is given  by $\|\cdot\|_{C(X)}$.
The  topological dual space of a Banach space $X$ will be denoted by $X'$, and the duality for a pair $X$ and $X'$ by $\langle\cdot,\cdot\rangle_{X'}$ or simply by $\langle\cdot,\cdot\rangle$
unless this leads to ambiguity.
Moreover, the letters $C$, $K$, $C_0$, $K_0$, $C_1$, $K_1$,..., denote positive constants, independent of state $(u,v)$ and control $f$, but its value may change from line to line.

In order to study the existence of solution of system (\ref{eq1})-(\ref{eq3}), we define the space 
\begin{equation*}\label{besov-4}
\widehat{W}^{2-2/p,p}(\Omega):=
\left\{\begin{array}{rcl}
W^{2-2/p,p}(\Omega)&\mbox{if}& p<3,\\
W^{2-2/p,p}_{\bf n}(\Omega)&\mbox{if}& p>3,
\end{array}
\right.
\end{equation*}
and we will often use the following regularity result for the heat equation
(see \cite[p. 344]{feireisl}).

\begin{lemma}\label{feireisl}
Let $1<p<+\infty$, $u_0\in\widehat{W}^{2-2/p,p}(\Omega)$ and  $g\in L^p(Q)$. Then the problem
\begin{equation*}
\left\{
\begin{array}{rcl}
\partial_tu-\Delta u&=&g \quad\mbox{ in }Q,\\
u(0,\cdot)&=&u_0\quad\mbox{ in }\Omega,\\
\dfrac{\partial u}{\partial{\bf n}}&=&0\quad\mbox{ on }(0,T)\times\partial\Omega,
\end{array}
\right.
\end{equation*}
admits a unique solution $u$ such that
\begin{equation*}
u\in  C([0,T];\widehat{W}^{2-2/p,p})\cap L^p(W^{2,p}),
\quad \partial_tu\in L^p(Q).
\end{equation*}
Moreover, there exists a positive constant $C:=C(p,\Omega,T)$ such that
\begin{equation*}\label{des-regularity}
\|u\|_{C(\widehat{W}^{2-2/p,p})}+\|\partial_tu\|_{L^p(Q)}+\|u\|_{L^p(W^{2,p})}
\le C(\|g\|_{L^p(Q)}+\|u_0\|_{\widehat{W}^{2-2/p,p}}).
\end{equation*}
\end{lemma}
For simplicity, in what follows we will use the following notation
\begin{equation*}\label{space_X}
X_p:=\{u\in C([0,T];\widehat{W}^{2-2/p,p})\cap L^p(W^{2,p}) :\, \partial_tu\in L^p(Q)\},
\end{equation*}
and its norm will be denoted by $\|\cdot\|_{X_{p}}$.  In fact, $u\in X_p$ iff 
$u\in W^{2,1}_p(\Omega):=\{u\in L^p(W^{2,p})\,:\, \partial_tu\in L^p(Q)\}$ and $u\in C([0,T];\widehat{W}^{2-2/p,p})$.

Throughtout this paper, we will use the following equivalent norms in $H^1(\Omega)$ and $H^2(\Omega)$,   respectively (see \cite{necas} for details):
\begin{eqnarray}
\|u\|^2_{H^1}&\simeq&\|\nabla u\|^2+\left(\int_\Omega u\right)^2,\quad\forall\, u\in H^1(\Omega),\label{equi}\\
\|u\|^2_{H^2}&\simeq& \Vert \Delta u \Vert^2+ \left(\displaystyle\int_{\Omega} u \right)^2, \quad\forall \, u\in H^2_{\bf n}(\Omega),\label{norma-1}
\end{eqnarray}
and the classical interpolation  inequality in $3D$ domains
\begin{equation}\label{interpol}
 \|u\|_{L^4}\le C\|u\|^{1/4}\|u\|^{3/4}_{H^1},\quad \forall\, u\in H^1(\Omega).
\end{equation} 

\begin{remark}
The problem (\ref{eq1})-(\ref{eq3}) is conservative in $u$, because the total mass $\int_\Omega u(t)$ remains constant in time. In fact, integrating  (\ref{eq1})$_1$  in $\Omega$ we have
\begin{equation*}\label{comp_u}
\frac{d}{dt}\left(\int_\Omega u\right)=0,\ \mbox{ i.e. }\ \int_\Omega u(t)=\int_\Omega u_0:=m_0,\quad \forall t>0. 
\end{equation*}
Also, integrating (\ref{eq1})$_2$ in $\Omega$ we deduce that $\int_\Omega v$ satisfies
$$
\frac{d}{dt}\left(\int_\Omega v\right)+\int_\Omega v= m_0+\int_\Omega f \,v\,\chi_{_{\Omega_c}} ,\quad \forall t>0.
$$
\end{remark}

\section{Existence of Weak Solutions of Problem (\ref{eq1})-(\ref{eq3})}
\label{sec:3}

\begin{definition}(Weak solution)\label{weak}
Let $f\in L^4(Q_c):=L^4(0,T;L^4(\Omega_c))$, $u_0\in L^2(\Omega)$, $v_0\in H^1(\Omega)$ with $u_0\ge 0$ and $v_0\ge 0$ in $\Omega$,  a pair $(u,v)$
is called weak solution of problem (\ref{eq1})-(\ref{eq3}) in $(0,T)$, if $u\ge0$, $v\ge 0$,
\begin{eqnarray}
&&u\in L^{5/3}(Q)\cap L^{5/4}(W^{1,5/4}),\  \partial_tu\in [L^{10}(W^{1,10})]',\label{st-1}\\
&&v\in L^\infty(H^1)\cap L^2(H^2),\ \partial_tv\in L^{5/3}(Q),\label{st-11}
\end{eqnarray}
the following variational formulation holds for the $u$-equation
\begin{equation}\label{st-2}
-\displaystyle\int_0^T\langle u,\partial_t\overline{u}\rangle+\int_0^T(\nabla u,\nabla \overline{u})+\int_0^T(u\nabla v,\nabla \overline{u})=(u_0,\overline{u}(0)),\ \forall\overline{u}\in \mathcal{X}_u,
\end{equation}
the $v$-equation (\ref{eq1})$_2$ holds pointwisely a.e.~$(t,x)\in Q$, and the initial and boundary conditions for $v$ (\ref{eq2})$_2$-(\ref{eq3})$_2$ are satisfied.
The space $\mathcal{X}_u$ given in (\ref{st-2}) is defined as follow
\begin{equation*}\label{space_weak}
\mathcal{X}_u=\{u\in L^{10}(W^{1,10})\,:\, \partial_t u\in L^{5/2}(Q)\mbox{ and }u(T)=0\mbox{ in }\Omega\}.
\end{equation*}
\end{definition} 
\begin{remark}\label{rmk-1}
This definition of weak solution implies, in particular, that 
$$
u\in L^\infty(L^1)\ \mbox{ and }\ \int_\Omega u(t)=\int_\Omega u_0=m_0.
$$
Also, each term of (\ref{st-2}) has sense. In particular, from \eqref{st-1}-\eqref{st-11} one has that $u\nabla v\in L^{10/9}(Q)$.
\end{remark}
\begin{theorem}(Existence of weak solutions of (\ref{eq1})-(\ref{eq3}))\label{weak_solution}
There exists a weak solution $(u,v)$ of system (\ref{eq1})-(\ref{eq3}) in the sense of Definition \ref{weak}. 
\end{theorem}
The proof of this theorem follows from the two next subsections. 

\subsection{Regularized Problem}

In order to prove Theorem \ref{weak_solution}, we will study the following family of regularized problems related to system (\ref{eq1})-(\ref{eq3}), for  any $\varepsilon\in(0,1)$.
Given an adequate regularization $(u_0^\varepsilon,v_0^\varepsilon)$ of initial data $(u_0,v_0)$, we define $(u^\varepsilon,z^\varepsilon)$ as the solution of

\begin{equation}\label{regg-1}
\left\{
\begin{array}{rcl}
\partial_tu^\varepsilon-\Delta u^\varepsilon&=&\nabla\cdot(u^\varepsilon\nabla v(z^\varepsilon))\quad \mbox{ in }Q,\\
\partial_tz^\varepsilon-\Delta z^\varepsilon+z^\varepsilon&=&u^\varepsilon+f\, v(z^\varepsilon)_+\chi_{_{\Omega_c}}\quad \mbox{ in } Q,\\
u^\varepsilon(0)=u^\varepsilon_0,\ z^\varepsilon(0)&=&v^\varepsilon_0-\varepsilon\Delta v_0^\varepsilon \quad \mbox{ in }\Omega\\
\dfrac{\partial u^\varepsilon}{\partial {\bf n}}=0,\ \dfrac{\partial z^\varepsilon}{\partial{\bf n}}&=&0\quad \mbox{ on }(0,T)\times\partial\Omega,
\end{array}
\right.
\end{equation}
where $v^\varepsilon:=v(z^\varepsilon)$ is the unique solution of the problem
\begin{equation}\label{regg-1-1}
\left\{
\begin{array}{rcl}
v^\varepsilon-\varepsilon\Delta v^\varepsilon&=&z^\varepsilon\quad\mbox{ in }\Omega,\\
\dfrac{\partial v^\varepsilon}{\partial{\bf n}}&=&0\quad \mbox{ on }\partial\Omega,
\end{array}
\right.
\end{equation}
and $v_+:=\max\{v,0\}\ge0$.

We choose the initial conditions $u^\varepsilon_0$ and $v^\varepsilon_0$, with $u_0^\varepsilon\ge 0$ in $\Omega$,  
such that $(u^\varepsilon_0,v_0^\varepsilon-\varepsilon\Delta v_0^\varepsilon)\in W^{4/5,5/3}(\Omega)\times W^{7/5,10/3}_{\bf n}(\Omega)$
and
\begin{equation}\label{reg1-1}
(u^\varepsilon_0,v_0^\varepsilon-\varepsilon\Delta v_0^\varepsilon)\rightarrow(u_0,v_0)\quad \mbox{in }L^2(\Omega)\times H^1(\Omega),\mbox{ as }\varepsilon\rightarrow0.
\end{equation}
In the remaining of this section, we will denote $v(z^\varepsilon)$ only by $v^\varepsilon$. 

\begin{definition}\label{sol-reg}
Let $u_0^\varepsilon\in W^{4/5,5/3}(\Omega)$, $v_0^\varepsilon-\varepsilon\Delta v_0^\varepsilon\in W^{7/5,10/3}_{{\bf n}}(\Omega)$ with $u_0^\varepsilon\ge 0$  in $\Omega$, 
and $f\in L^4(Q_c)$. We say that a pair $(u^\varepsilon,z^\varepsilon)$ is a (strong) solution of problem (\ref{regg-1}) in $(0,T)$, if
$u^\varepsilon\ge 0$ in $Q$, 
\begin{equation*}\label{sol-reg-1}
(u^\varepsilon,z^\varepsilon)\in X_{5/3}\times X_{10/3},
\end{equation*}
the  equations (\ref{regg-1})$_1$-(\ref{regg-1})$_2$ holds pointwisely a.e.~$(t,x)\in Q$, and the initial and boundary  conditions  (\ref{regg-1})$_3$-(\ref{regg-1})$_4$ are satisfied.
\end{definition}
\begin{remark}\label{compor}
Integrating  (\ref{regg-1})$_1$ in $\Omega$ we have
\begin{equation}\label{reg1-2}
\int_\Omega u^\varepsilon(t)=\int_\Omega u_0^\varepsilon:={m}^\varepsilon_0 \quad \forall t>0.
\end{equation}
In fact, $\|u^\varepsilon(t)\|_{L^1}=\|u^\varepsilon_0\|_{L^1}:=m^\varepsilon_0$. Moreover, integrating (\ref{regg-1})$_2$ in $\Omega$ we deduce 
$$
\frac{d}{dt}\left(\int_\Omega z^\varepsilon\right)+\int_\Omega z^\varepsilon={m}^\varepsilon_0+\int_\Omega f\,v^\varepsilon_+\chi_{_{\Omega_c}},
$$
which implies
\begin{equation*}\label{reg1-2-1}
\frac{d}{dt}\left(\int_\Omega z^\varepsilon\right)^2+\left(\int_\Omega z^\varepsilon\right)^2\le \left(m_0^\varepsilon+\int_\Omega f\,v^\varepsilon_+\chi_{_{\Omega_c}}\right)^2.
\end{equation*}
\end{remark}
\begin{theorem}\label{teo_reg}
There exists a strong solution $(u^\varepsilon,z^\varepsilon)\in X_{5/3}\times X_{10/3}$ of system (\ref{regg-1}) in $(0,T)$ in the sense of Definition \ref{sol-reg}.
\end{theorem}

The proof of Theorem \ref{teo_reg} is carried out in the Appendix.

\subsection{Proof of Theorem \ref{weak_solution}. Taking limit as $\varepsilon\rightarrow0.$}


From the energy inequality (\ref{F-6}) (see the proof of Lemma \ref{fix} in the Appendix) and the conservativity  property (\ref{reg1-2}) we deduce the following estimates
(uniformly with respect to $\varepsilon$)
\begin{equation}\label{cot-1}
\left\{
\begin{array}{l}
\{\nabla\sqrt{u^\varepsilon+1}\}_{\varepsilon>0}\quad\mbox{ is bounded in }L^2(Q),\\
\{\sqrt{u^\varepsilon+1}\}_{\varepsilon>0}\quad\mbox{ is bounded in }L^\infty(L^2)\cap L^2(L^6)\hookrightarrow L^{10/3}(Q)\cap L^8(L^{12/5}),\\
\{v^\varepsilon\}_{\varepsilon>0}\quad\mbox{ is bounded in }L^\infty(H^1)\cap L^2(H^2),\\
\{\sqrt\varepsilon\Delta v^\varepsilon\}_{\varepsilon>0}\quad\mbox{ is bounded in }L^\infty(L^2)\cap L^2(H^1),
\end{array}
\right.
\end{equation}
which implies 
\begin{equation}\label{cot-2}
\left\{
\begin{array}{l}
\{u^\varepsilon\}_{\varepsilon>0}\quad\mbox{ is bounded in } L^{5/3}(Q)\cap L^4(L^{6/5}),\\
\{z^\varepsilon\}_{\varepsilon>0}\quad\mbox{ is bounded in }L^\infty(L^2)\cap L^2(H^1),\\
\{\partial_tu^\varepsilon\}_{\varepsilon>0}\quad\mbox{ is bounded in }[L^{10}(W^{1,10})]',\\
\{\partial_tz^\varepsilon\}_{\varepsilon>0}\quad\mbox{ is bounded in }[L^{2}(H^1)]'.
\end{array}
\right.
\end{equation}
On the other hand, taking into account that $\nabla u^\varepsilon=2\sqrt{u^\varepsilon+1}\nabla\sqrt{u^\varepsilon+1}$, from (\ref{cot-1})$_1$ and (\ref{cot-1})$_2$ we deduce that 
\begin{equation}\label{cot-4}
\{u^\varepsilon\}_{\varepsilon>0}\quad\mbox{ is bounded in }L^{5/4}(W^{1,5/4}).
\end{equation}
Also, from (\ref{cot-1})$_3$ we have that $\{\nabla v^\varepsilon\}_{\varepsilon>0}$ is bounded in $L^\infty(L^2)\cap L^2(H^1)\hookrightarrow L^{10/3}(Q)$, which
jointly to (\ref{cot-2})$_1$ implies that
\begin{equation}\label{cot-5}
\{u^\varepsilon\nabla v^\varepsilon\}_{\varepsilon>0}\quad\mbox{ is bounded in }L^{10/9}(Q).
\end{equation}
Notice that from (\ref{regg-1-1}) and (\ref{cot-1})$_4$ we obtain that 
\begin{equation}\label{dif-v-z}
z^\varepsilon-v^\varepsilon=-\varepsilon\Delta v^\varepsilon\rightarrow0 \quad \hbox{as $\varepsilon\rightarrow 0$, \quad in the $L^\infty(L^2)\cap L^2(H^1)$-norm.}
\end{equation}
Therefore, from (\ref{cot-1}), (\ref{cot-2}), (\ref{cot-4}) and (\ref{dif-v-z}),  we deduce that there exists limit functions $(u,v)$ such that 
$$
\left\{
\begin{array}{l}
u\in L^{5/3}(Q)\cap L^{5/4}(W^{1,5/4}),\\
v\in L^\infty(H^1)\cap L^2(H^2),
\end{array}
\right.
$$
and for some subsequence of $\{(u^\varepsilon,v^\varepsilon,z^\varepsilon)\}_{\varepsilon>0}$, still denoted by 
$\{(u^\varepsilon,v^\varepsilon,z^\varepsilon)\}_{\varepsilon>0}$, the following convergences holds, as $\varepsilon\rightarrow0$,
\begin{equation}\label{cot-7}
\left\{
\begin{array}{rcl}
u^\varepsilon&\rightarrow& u\quad\mbox{ weakly in } L^{5/3}(Q)\cap L^{5/4}(W^{1,5/4}),\\
v^\varepsilon&\rightarrow& v \quad\mbox{ weakly in }L^2(H^2)\mbox{ and weakly* in }L^\infty(H^1),\\
z^\varepsilon&\rightarrow& v \quad \mbox{ weakly in }L^2(H^1)\mbox{ and weakly* in }L^\infty(L^2),\\
\partial_tu^\varepsilon&\rightarrow&\partial_tu \quad\mbox{ weakly* in }[L^{10}(W^{1,10})]',\\
\partial_tz^\varepsilon&\rightarrow&\partial_tv \quad\mbox{ weakly* in }[L^{2}(H^1)]'.
\end{array}
\right.
\end{equation}
We will verify that $(u,v)$ is a weak solution of (\ref{eq1})-(\ref{eq3}). 
From (\ref{cot-2})$_3$, (\ref{cot-4}) and the Aubin-Lions lemma (see \cite[Th\'eor\`eme 5.1, p. 58]{lions}) we deduce that 
\begin{equation}\label{cot-6}
\{u^\varepsilon\}_{\varepsilon>0} \mbox{ is relatively compact in }L^{5/4}(L^2)\ (\mbox{and also in }L^p(Q),\, \forall p<5/3 ). 
\end{equation}
Thus, from (\ref{cot-7})$_2$, (\ref{cot-6}) and taking into account (\ref{cot-5}) we have
\begin{equation}\label{cot-9}
u^\varepsilon\nabla v^\varepsilon\rightarrow u\nabla v \quad \mbox{ weakly in } L^{10/9}(Q).
\end{equation}

On the other hand, from (\ref{cot-7})$_3$, (\ref{cot-7})$_5$,  \cite[Th\'eor\`eme 5.1, p. 58]{lions} and \cite[Corollary 4]{simon} we obtain
\begin{equation}\label{compact-z}
z^\varepsilon\rightarrow v\ \mbox{ strongly in }\ L^2(Q)\cap C([0,T];(H^1)').
\end{equation}
Thus, from (\ref{dif-v-z}), (\ref{cot-7})$_2$ and (\ref{compact-z})  we deduce that $v^\varepsilon$ converges to $v$  strongly  in $L^2(Q)$, which implies
$$
v^\varepsilon_+\rightarrow v_+\ \mbox{ strongly in }L^2(Q).
$$
Then, using that $\{v^\varepsilon\}_{\varepsilon>0}$ is bounded in $L^\infty(H^1)\cap L^2(H^2)\hookrightarrow L^{10}(Q)$ and $f\in L^4(Q_c)$, we deduce 
\begin{equation}\label{cot-10}
f\,v^\varepsilon_+ \chi_{_{\Omega_c}}\rightarrow f\,v_+ \chi_{_{\Omega_c}} \quad \mbox{ weakly in }L^{20/7}(Q).
\end{equation}
Also from (\ref{compact-z}),  $z^\varepsilon(0)$ converges to $v(0)$ in $H^1(\Omega)'$, then from (\ref{reg1-1}) and the uniqueness of the limit we have $v(0)=v_0$, which is the initial condition given in (\ref{eq2})$_2$.

Therefore, taking to the limit   in the regularized problem (\ref{regg-1}), as $\varepsilon\rightarrow0$, and taking into account (\ref{reg1-1}), (\ref{cot-7}), (\ref{cot-9}) and (\ref{cot-10}) we conclude that $(u,v)$ satisfies the weak formulation
\begin{eqnarray}
&&
-\int_0^T\langle u,\partial_t\overline{u}\rangle+\int_0^T(\nabla u,\nabla\overline{u})+\int_0^T(u\nabla v, \nabla\overline{u})=(u_0,\overline{u}(0))\quad \forall \,\overline{u} \in \mathcal{X}_u,\label{cot-11}\vspace{0.3cm}\\
&&\displaystyle\int_0^T\langle\partial_tv,\overline{z}\rangle+\int_0^T(\nabla v,\nabla\overline{z})+\int_0^T(v,\bar{z})
=\int_0^T(u,\overline{z})+\int_0^T(f\,v_+\chi_{_{\Omega_c}},\overline{z})
\quad \forall\, \overline{z}\in L^{2}(H^1).\label{cot-12}
\end{eqnarray}
Integrating by parts in (\ref{cot-12}), and using that $u\in L^{5/3}(Q)$ and $v\in L^2(H^2)$, we deduce that $v$ is the unique solution of the problem 
\begin{equation}\label{cot-13}
\left\{
\begin{array}{rcl}
\partial_tv-\Delta v+v&=&u+f\,v_+\chi_{_{\Omega_c}}\quad\mbox{ in } L^{5/3}(Q),\\
v(0)&=&v_0\quad\mbox{ in }\Omega,\\
\dfrac{\partial v}{\partial{\bf n}}&=&0\quad\mbox{ on }(0,T)\times\partial\Omega.
\end{array}
\right.
\end{equation}
Finally, we will check the positivity of  $(u,v)$.  Indeed, the positivity of $u$  follow from (\ref{cot-6}) and the fact that $u^\varepsilon\ge0$ a.e. $(t,x)\in Q$ (see Lemma \ref{fix} in the Appendix). 
In order to check that $v\ge0$, we test (\ref{cot-13})$_1$ by $v_-:=\min\{v,0\}\le0$, taking into account that $u\ge 0$, and using
that $v_-=0$ if $v\ge0$, $\nabla v_-=\nabla v$ if $v\le0$ and $\nabla v_-=0$ if $v>0$, we obtain
$$
\frac12\frac{d}{dt}\|v_-\|^2+\|\nabla v_-\|^2+\|v_-\|^2=(u,v_-)+(f\,v_+\chi_{_{\Omega_c}},v_-)\le0,
$$ 
which implies that $v_-\equiv0$, then $v\ge0$ a.e. $(t,x)\in Q$. Thus, since $v_+\equiv v$ then $v\ge0$ is also a solution of the $v$-equation (\ref{eq1})$_2$.

\section{Regularity Criterion}
\label{sec:4}

In this section we will give a regularity criterion  of  system (\ref{eq1})-(\ref{eq3}).

\begin{definition}(Strong solution of problem (\ref{eq1})-(\ref{eq3}))\label{regul}
Let $f\in L^4(Q_c)$, $u_0\in H^1(\Omega)$, $v_0\in W^{3/2,4}_{{\bf n}}(\Omega)$ with $u_0\ge 0$ and $v_0\ge 0$ in $\Omega$. A pair $(u,v)$ is called strong solution of problem
(\ref{eq1})-(\ref{eq3}) in (0,T), if $u\ge0$, $v\ge 0$ in $Q$,
\begin{eqnarray}\label{eqq2}
(u,v)\in X_2\times X_4,
\end{eqnarray}
the  system (\ref{eq1}) holds pointwisely a.e.~$(t,x)\in Q$, and the initial and boundary   conditions (\ref{eq2}) and (\ref{eq3}) are satisfied.
\end{definition}
\begin{remark}
Using the interpolation inequality (\ref{interpol}), Gronwall lemma and proceeding as for the Navier-Stokes
equations (see \cite{temam}), we can deduce the uniqueness of strong solutions of (\ref{eq1})-(\ref{eq3}).
\end{remark}
\begin{theorem}(Regularity Criterion)\label{regularity}
Let $(u,v)$ be a weak solution of (\ref{eq1})-(\ref{eq3}). If, in addition,
$u_0\in H^1(\Omega)$, $v_0\in W^{3/2,4}_{{\bf n}}(\Omega)$ 
 and the following regularity criterion holds
\begin{equation}\label{hyp-regul}
u\in L^{20/7}(Q),
\end{equation}
then $(u,v)$ is a strong solution
of (\ref{eq1})-(\ref{eq3}) in sense of Definition \ref{regul}. Moreover, there exists a positive constant 
$K=K(\|u_0\|_{H^1},\|v_0\|_{W_{\bf n}^{3/2,4}},\|f\|_{L^4(Q)})$ such that
\begin{equation}\label{bound-strong}
\|u,v\|_{X_2\times X_4}\le K.
\end{equation}
\end{theorem}
The proof of this theorem follows from the two next subsections.

\subsection{Interpolation and embedding results}\label{ss-ie}
In order to proof  Theorem \ref{regularity}, starting from the regularity of $u$ and $v$, we will get the regularity for $\nabla \cdot (u\nabla v)$ which improves the regularity for $u$. 
With this new regularity for $u$, the regularity for $\nabla \cdot (u\nabla v)$ is improved several times using a bootstraping argument. 
Along the proof of Theorem \ref{regularity}, different interpolation results  will be used together with some embeddings results that will be stated below.

As a consequence of the interpolation inequality
$$
\|u\|_{L^p}\le\|u\|^{1-\theta}_{L^{p_1}}\|u\|^\theta_{L^{p_2}},\ \mbox{ with }\ \frac1p=\frac{1-\theta}{p_1}+\frac{\theta}{p_2}\ \mbox{ and }\ \theta\in[0,1]
$$
we have the following result
\begin{lemma}\label{l3}
Let $p_1,p_2,q_1,q_2,p,q\ge1$ such that 
$$
\frac1q=\frac{1-\theta}{q_1}+\frac{\theta}{q_2}\ \mbox{ and }\ \frac1p=\frac{1-\theta}{p_1}+\frac{\theta}{p_2}, \mbox{ with }\theta\in[0,1].
$$
Then,
\begin{equation}\label{l5-e1}
L^{p_1}(L^{q_1}) \cap L^{p_2}(L^{q_2}) \hookrightarrow L^p(L^q).
\end{equation}
\end{lemma}




Using the Sobolev embedding 
$$
W^{r,p}(\Omega)\hookrightarrow L^q(\Omega),\ \mbox{ with }\ \frac1q=\frac1p-\frac rN,
$$
where $N$ is the space-dimension and the Gagliardo-Nirenberg inequality (see \cite[Theorem 10.1]{friedman})
$$
W^{s,p_1}(\Omega)\cap L^{p_2}(\Omega)\hookrightarrow L^p(\Omega),\ \mbox{ with }\ \frac1p=\theta\left(\frac{1}{p_1}-\frac sN\right)+\frac{1-\theta}{p_2}\mbox{ and }\theta\in[0,1]
$$
we deduce the following result
\begin{lemma}\label{l4}
Let $p_1,q_1,p_2,p,q\ge1$  such that 
$$
\frac{1}{q}=\frac{1-\theta}{q_1}+\theta\left(\frac{1}{p_1}-\frac{r}{N} \right) \mbox{ and }
\frac1p=\frac{\theta}{p_2}  \mbox{ with $\theta\in[0,1]$ and $r>0$}.
$$ 
Then,
\begin{equation*}\label{l4-e1}
L^{\infty}(L^{q_1}) \cap L^{p_2}(W^{r,p_1}) \hookrightarrow L^p(L^q).
\end{equation*}
\end{lemma}


\begin{lemma} \label{l6} (\cite[Theorem 7.58, p.218]{adams})
Let $1<p<2$, and $r, s>0$   such that  
$$
s=N\left( \frac{1}{2}-\frac{1}{p} \right)+r.
$$ 
Then,
\begin{equation*}\label{l6-e1}
W^{r,p}(\Omega) \hookrightarrow H^s(\Omega).
\end{equation*}
\end{lemma}

\begin{lemma} \label{l7} (\cite[Th\'eor\`{e}me 9.6, p. 49]{lions-magenes})
Let $p_1,p_2,p\ge 1$ and $s_1,s_2,s>0$ such that
$$
s=(1-\theta)s_1 + \theta s_2 \ \mbox{ and }\  \frac1p=\frac{1-\theta}{p_1}+\frac{ \theta}{p_2},\mbox{ with }\theta\in[0,1].
$$ 
Then,
\begin{equation*}\label{l7-e1}
L^{p_1}(H^{s_1}) \cap L^{p_2}(H^{s_2}) \hookrightarrow L^p(H^s).
\end{equation*}
\end{lemma}

\subsection{Proof of Theorem \ref{regularity}}\label{ss-p}

\begin{proof} 

The proof is carried out into four steps:

\vspace{0.3cm}
 \underline{Step 1:}\hspace{0.2cm} $v\in X_{20/7}$
\vspace{0.3cm}

\noindent From Theorem \ref{weak_solution}, we know that there exists a weak solution $(u,v)$ of system (\ref{eq1})-(\ref{eq3}) in the sense 
of Definition \ref{weak}. Thus, in particular $v\in L^{10}(Q)$ and then
$fv\chi_{_{\Omega_c}}\in L^{20/7}(Q)$, which implies, using  hypothesis (\ref{hyp-regul}), that $u+fv\chi_{_{\Omega_c}}\in L^{20/7}(Q)$. Then, 
applying Lemma \ref{feireisl} (for $p={20}/{7}$) to equation (\ref{eq1})$_2$, we have $v\in X_{20/7}$.
In particular, using Sobolev embeddings we have

\begin{equation}\label{ecu-2-a}
v\in L^\infty(Q),
\end{equation}
\begin{equation}\label{ecu-2}
\nabla v\in L^\infty(L^4)\cap L^{20/7}(W^{1,20/7})\hookrightarrow L^\infty(L^4)\cap L^{20/7}(L^{60}).
\end{equation}
Embedding (\ref{l5-e1}) for $p_1=\infty$, $q_1=4$, $p_2=20/7$ and $q_2=60$  (see Lemma \ref{l3}) implies $p=q={20}/{3}$ hence
\begin{equation}\label{ecu-2-b}
\nabla v\in L^\infty(L^4)\cap L^{20/7}(L^{60})\hookrightarrow L^{20/3}(Q).
\end{equation}

\vspace{0.3cm}
\underline{Step 2:}\hspace{0.3cm} $u\in L^\infty(L^2)\cap L^2(H^1)$.
\vspace{0.3cm}

Starting from $u\in L^{20/7}(Q) \cap L^{5/4}(W^{1,5/4})$ and $v \in X_{20/7}$, we improve the regularity of $u$ by a bootstrapping argument  in eigth sub-steps:

{\sl i) $u \in X_{20/19}$:}

\noindent 
Using that $(u,\Delta v)\in L^{20/7}(Q)\times L^{20/7}(Q)$ (hence $u\Delta v\in L^{10/7}(Q)$), and $(\nabla u,\nabla v)\in L^{5/4}(Q)\times L^{20/3}(Q)$ (hence $\nabla u\cdot\nabla v\in L^{20/19}(Q)$) we have
$$
\nabla\cdot(u\nabla v)=u\Delta v+\nabla u\cdot\nabla v\in L^{20/19}(Q).
$$
Thus, applying Lemma \ref{feireisl} (for $p=20/19$) to equation (\ref{eq1})$_1$ we obtain that $u\in X_{20/19}$.

\bigskip

{\sl ii) $u \in X_{10/9}$:}
Since $u\in X_{20/19}$, then by Sobolev embeddings
\begin{equation}\label{ecu-2-1}
\nabla u\in L^{20/19}(W^{1,20/19})\hookrightarrow L^{20/19}(L^{60/37}).
\end{equation}
Moreover, 
 using (\ref{l5-e1}) in (\ref{ecu-2}) (for $p_1=\infty$, $q_1=4$, $p_2=20/7$, $q_2=60$ and $p=20$, hence $q={60}/{13}$), we obtain
\begin{equation}\label{ecu-2-3}
\nabla v\in L^{\infty}(L^4)\cap L^{20}(L^{60/13}).
\end{equation}
Thus, from (\ref{ecu-2-1}) and (\ref{ecu-2-3}) we have 
$\nabla u\cdot\nabla v\in L^{20/19}(L^{15/13})\cap L^1(L^{6/5})$. Then, 
owing to (\ref{l5-e1}) applied to $(p_1,q_1)=\left({20}/{19},{15}/{13}\right)$ and $(p_2,q_2)=\left(1,{6}/{5}\right)$
implies that $p=q=10/9$, hence
\begin{equation*}\label{e-n1}
\nabla u\cdot\nabla v\in L^{10/9}(Q).
\end{equation*} 
Since $u\Delta v\in L^{10/7}(Q)$, we have $\nabla\cdot(u\nabla v)\in L^{10/9}(Q)$. Then, applying Lemma \ref{feireisl} (for $p=10/9$)
to  (\ref{eq1})$_1$ we deduce that $u\in X_{10/9}$. 

\bigskip

{\sl iii) $u \in X_{20/17}$:}
Since $u\in X_{10/9}$, then
\begin{equation}\label{ecu-2-4}
\nabla u\in L^{10/9}(W^{1,10/9})\hookrightarrow L^{10/9}(L^{30/17}).
\end{equation}
Now, using (\ref{l5-e1}) in (\ref{ecu-2}) (for $p_1=\infty$, $q_1=4$, $p_2=20/7$, $q_2=60$ and $p=10$, hence $q={60}/{11}$), we obtain
\begin{equation*}\label{nn1}
\nabla v\in L^\infty(L^4)\cap L^{10}(L^{60/11}),
\end{equation*} 
which jointly to (\ref{ecu-2-4}) implies 
$\nabla u\cdot\nabla v\in L^{10/9}(L^{60/49})\cap L^1(L^{4/3})$. Then  using (\ref{l5-e1}) with $(p_1,q_1)=\left({10}/{9},{60}/{49}\right)$, $(p_2,q_2)=\left(1,{4}/{3}\right)$
implies that $p=q=20/17$, hence
\begin{equation*}\label{e-n2}
\nabla u\cdot\nabla v\in L^{20/17}(Q).
\end{equation*}
Since 
$u\Delta v\in L^{10/7}(Q)$, we have $\nabla\cdot(u\nabla v)\in L^{20/17}(Q)$.  Then, applying Lemma \ref{feireisl} (for $p=20/17$) to  (\ref{eq1})$_1$ we deduce
that $u\in X_{20/17}$. 

\bigskip

{\sl iv) $u \in X_{5/4}$:} Since $u\in X_{20/7}$ then
$$
\nabla u\in L^{20/17}(W^{1,20/17})\hookrightarrow L^{20/17}(L^{60/31}),
$$
and,  from (\ref{ecu-2-b}), $\nabla v\in L^\infty(L^4)\cap L^{20/3}(Q)$, 
 then $\nabla u\cdot\nabla v\in L^{20/17}(L^{30/23})\cap L^1(L^{3/2})$, which thanks to (\ref{l5-e1}) applied to $(p_1,q_1)=\left({20}/{17},{30}/{23}\right)$, $(p_2,q_2)=\left(1,{3}/{2}\right)$
implies $p=q=5/4$ hence
 \begin{equation*}\label{e-n3}
 \nabla u\cdot\nabla v\in L^{5/4}(Q).
 \end{equation*}
 Since  $u\Delta v\in L^{10/7}(Q)$, we obtain that $\nabla\cdot(u\nabla v)\in L^{5/4}(Q)$ and, applying Lemma \ref{feireisl} (for $p=5/4$) to equation (\ref{eq1})$_1$ we deduce 
$u\in X_{5/4}$.

\bigskip

{\sl v) $u \in X_{4/3}$:} Using that $u\in X_{5/4}$, then
\begin{equation}\label{ecu-2-5}
\nabla u\in L^{5/4}(W^{1,5/4})\hookrightarrow L^{5/4}(L^{15/7}).
\end{equation}
 Using (\ref{l5-e1}) in (\ref{ecu-2}) (for $p_1=\infty$, $q_1=4$, $p_2=20/7$, $q_2=60$ and $p=5$, hence $q={60}/{7}$), we obtain
\begin{equation*}\label{nn3}
\nabla v\in L^\infty(L^4)\cap L^5(L^{60/7});
\end{equation*} 
then from the latter regularity and (\ref{ecu-2-5}) we have
$\nabla u\cdot\nabla v\in L^{5/4}(L^{60/43})\cap L^1(L^{12/7})$,
which thanks to (\ref{l5-e1}) applied to $(p_1,q_1)=\left({5}/{4},{60}/{43}\right)$, $(p_2,q_2)=\left(1,2\right)$
implies $p=q=4/3$, hence
\begin{equation*}\label{e-n4}
\nabla u\cdot\nabla v\in L^{4/3}(Q).
\end{equation*} 
Since $u\Delta v\in L^{10/7}(Q)$, we obtain
$\nabla\cdot(u\nabla v)\in L^{4/3}(Q)$. Then, applying Lemma \ref{feireisl} to equation (\ref{eq1})$_1$ we have $u\in X_{4/3}$.

\bigskip

{\sl vi) $u \in X_{10/7}$:} Since $u\in X_{4/3}$, then
$$
\nabla u\in L^{4/3}(W^{1,4/3})\hookrightarrow L^{4/3}(L^{12/5}),
$$
again using (\ref{l5-e1}) in (\ref{ecu-2}) (for $p_1=\infty$, $q_1=4$, $p_2=20/7$, $q_2=60$ and $p=4$, hence $q=12$), we obtain
\begin{equation*}\label{nn4}
\nabla v\in L^\infty(L^4)\cap L^4(L^{12})
\end{equation*} 
and $
\nabla u\cdot\nabla v\in L^{4/3}(L^{3/2})\cap L^1(L^2)$,
which thanks to (\ref{l5-e1}) applied to $(p_1,q_1)=\left({4}/{3},{3}/{2}\right)$, $(p_2,q_2)=\left(1,2\right)$
implies $p=q=10/7$, hence
\begin{equation*}\label{e-n5}
\nabla u\cdot\nabla v\in L^{10/7}(Q).
\end{equation*}
 Since $u\Delta v\in L^{10/7}(Q)$, we obtain $\nabla\cdot(u\nabla v)\in L^{10/7}(Q)$, and applying Lemma \ref{feireisl} (for $p=10/7$) to equation (\ref{eq1})$_1$ we have $u\in X_{10/7}$. 

\bigskip

{\sl vii) $u \in X_{20/13}$:} Since $u\in X_{10/7}$, then  
\begin{equation}\label{ecu-2-6}
\left\{
\begin{array}{l}
u\in L^\infty(W^{3/5,10/7})\cap L^{10/7}(W^{2,10/7})\hookrightarrow L^\infty(L^2)\cap L^{10/7}(L^{30})\hookrightarrow L^{10/3}(Q),\\
\nabla u\in L^{10/7}(W^{1,10/7})\hookrightarrow L^{10/7}(L^{30/11}).
\end{array}
\right.
\end{equation}
 This time, we use (\ref{l5-e1}) in (\ref{ecu-2}) (for  $p_1=\infty$, $q_1=4$, $p_2=20/7$, $q_2=60$ and $p=10/3$, hence $q=20$), we obtain
\begin{equation*}\label{nn5}
\nabla v\in L^\infty(L^4)\cap L^{10/3}(L^{20}),
\end{equation*} 
the latter regularity, (\ref{ecu-2-6}) 
and the fact that $\Delta v\in L^{20/7}(Q)$ implies
$$
u\Delta v\in L^{20/13}(Q)\ \mbox{ and }\ \nabla u\cdot\nabla v\in L^{10/7}(L^{60/37})\cap L^1(L^{12/5}).
$$
From (\ref{l5-e1}) applied to $(p_1,q_1)=\left({10}/{7},{60}/{37}\right)$, $(p_2,q_2)=\left(1,{12}/{5}\right)$ one has $p=q=20/13$ hence
\begin{equation*}\label{e-n6}
\nabla u\cdot\nabla v\in L^{20/13}(Q).
\end{equation*}
Then, applying Lemma \ref{feireisl} (for $p=20/13$) to equation (\ref{eq1})$_1$ we have 
$u\in X_{20/13}$.

\bigskip

{\sl viii) $u \in L^\infty(L^2)\cap L^2(H^1)$:
From Lemma \ref{l6},  we know that $W^{7/10,20/13}(\Omega)\hookrightarrow H^{1/4}(\Omega)$ and $W^{2,20/13}(\Omega)\hookrightarrow H^{31/20}(\Omega)$. Therefore, from $u\in X_{20/13}$
we can deduce 
\begin{equation*}\label{ecu-2-8}
u\in L^\infty(H^{1/4})\cap L^{20/13}(H^{31/20}).
\end{equation*}

Moreover, from Lemma \ref{l7} for $(p_1,s_1)=\left(\infty,{1}/{4}\right)$, $(p_2,s_2)=\left({20}/{13},{31}/{20}\right)$ we have
that $u\in L^2(H^{5/4})\hookrightarrow L^2(H^1)$.} Therefore, from the latter regularity and (\ref{ecu-2-6})$_1$
we deduce 
\begin{equation}\label{ecu-2-9}
u\in L^\infty(L^2)\cap L^2(H^1)\hookrightarrow L^{10/3}(Q).
\end{equation}

\vspace{0.3cm}
\underline{Step 3:}\hspace{0.3cm} $(u,v)\in X_{5/3}\times X_{10/3}$, $u\in L^5(Q)$  and $\nabla u\in L^{20/9}(Q)$.
\vspace{0.3cm}

\noindent From  (\ref{ecu-2-a}), (\ref{ecu-2-9}) and the fact that $f\in L^4(Q)$ we obtain  $u+fv\in L^{10/3}(Q)$. Then applying Lemma \ref{feireisl} (for $p=10/3$) to
equation (\ref{eq1})$_2$ we have that $v\in X_{10/3}$.
In particular, from Lemma \ref{l4} (for $p_1=p_2=10/3$, $q_1=6$, $r=1$ and $p=q=10$)
we obtain
$\nabla v\in L^\infty(L^6)\cap L^{10/3}(W^{1,10/3})\hookrightarrow L^{10}(Q).$ Then, using that $(u,\Delta v)\in L^{10/3}(Q)\times L^{10/3}(Q)$, $\nabla v\in L^{10}(Q)$ and
taking into account that $\nabla u\in L^2(Q)$ we have 
$$
\nabla\cdot(u\nabla v)=u\Delta v+\nabla u\cdot\nabla v\in L^{5/3}(Q).
$$
Thus, applying Lemma \ref{feireisl} (for $p=5/3$) to equation (\ref{eq1})$_1$ we obtain that
$ u\in X_{5/3}$. 
Moreover, 
from Sobolev embeddings and again  Lemma \ref{l4} (for $p_1=p_2=5/3$, $q_1=3$, $r=2$ and $p=q=5$)
 we have
\begin{equation}\label{ecu-8}
u\in L^\infty(L^3)\cap L^{5/3}(W^{2,5/3})\hookrightarrow L^5(Q).
\end{equation}

From Lemma \ref{l6} we have the embeddings
$W^{4/5,5/3}(\Omega)\hookrightarrow H^{1/2}(\Omega)$ and $W^{2,5/3}(\Omega)\hookrightarrow H^{17/10}(\Omega)$. Thus, since $u\in X_{5/3}$, one has
\begin{equation*}\label{ecu-9}
u\in L^\infty(H^{1/2})\cap L^{5/3}(H^{17/10}).
\end{equation*}

Moreover, from 
Lemma \ref{l7} (for $(p_1,s_1)=\left(\infty,{1}/{2}\right)$ and $(p_2,s_2)=\left({5}/{3},{17}/{10}\right)$),
we have 
$u\in L^{20/9}(H^{7/5})$, and in particular  $\nabla u\in L^{20/9}(H^{2/5})\hookrightarrow L^{20/9}(Q)$.

\vspace{0.3cm}
\underline{Step 4:}\hspace{0.3cm} $(u,v)\in X_{2}\times X_{4}$.
\vspace{0.3cm}

\noindent From (\ref{ecu-2-a}), (\ref{ecu-8}), and using  that $f\in L^4(Q_c)$, we have $u+fv\chi_{_{\Omega_c}}\in L^4(Q)$. Then, applying Lemma~\ref{feireisl} (for $p=4$) to equation
(\ref{eq1})$_2$ we deduce that $v\in X_4$ and satisfies the estimate
\begin{eqnarray}
\|v\|_{X_{4}}
&\le& C(\|u+fv\|_{L^{4}(Q)}+\|v_0\|_{W^{3/2,4}_{\bf n}})\le C(\|u\|_{L^{4}(Q)}+\|f\|_{L^4(Q)}\|v\|_{L^{\infty}(Q)}+\|v_0\|_{W^{3/2,4}_{\bf n}})\nonumber\\
&\le&C_0(\|u_0\|_{W^{4/5,5/3}},\|v_0\|_{W^{3/2,4}_{\bf n}},\|f\|_{L^4(Q)})\label{ecu-11}.
\end{eqnarray}
In particular, by Sobolev embeddings  and   Lemma \ref{l4} 
(for $p_1=p_2=4$, $q_1=12$, $r=1$ hence $p=q=20$) we have $\nabla v\in L^\infty(L^{12})\cap L^4(W^{1,4})\hookrightarrow L^{20}(Q)$.

Now, using that $(u,\Delta v)\in L^5(Q)\times L^4(Q)$ and $(\nabla u,\nabla v)\in L^{20/9}(Q)\times L^{20}(Q)$ we obtain
$$
\nabla\cdot(u\nabla v)=u\Delta v+\nabla u\cdot\nabla v\in L^2(Q).
$$
Therefore, applying Lemma \ref{feireisl} (for $p=2$) to equation (\ref{eq1})$_1$ we deduce that $u\in X_2$ and 
\begin{eqnarray}\label{ecu-12}
\|u\|_{X_2}&\le&C(\|u\|_{L^5(Q)}\|\Delta v\|_{L^4(Q)}+\|\nabla u\|_{L^{20/9}(Q)}\|\nabla v\|_{L^{20}(Q)}+\|u_0\|_{H^1})\nonumber\\
&\le&C_1(\|u_0\|_{H^1},\|v_0\|_{W^{3/2,4}_{\bf n}},\|f\|_{L^4(Q)}).
\end{eqnarray}
Finally, we observe that estimate (\ref{bound-strong}) follows from (\ref{ecu-11}) and (\ref{ecu-12}).
\end{proof}

\section{The Optimal Control Problem}
\label{sec:5}

In this section we establish the statement of the bilinear control problem. Following \cite{casas-1,casas-2}, we formulate the control problem in such a way
that any admissible state is a strong solution of (\ref{eq1})-(\ref{eq3}). Since there is no existence result of global in time strong solutions of (\ref{eq1})-(\ref{eq3}), we have to choose
a suitable objective functional.

We suppose that 
\begin{equation}\label{F-def}
\mathcal{F}\subset L^4(Q_c):=L^4(0,T;L^4(\Omega_c)) \quad \mbox{ is a 
nonempty and convex set,} 
\end{equation}
where $\Omega_c\subset\Omega$ is the control domain. We consider data 
$u_0\in H^1(\Omega)$, $v_0\in W^{3/2,4}_{{\bf n}}(\Omega)$ with $u_0\ge0$ and $v_0\ge 0$ in $\Omega$, and  the function $f\in\mathcal{F}$ describing the bilinear control acting on the  $v$-equation. 

Now, we define the following constrained minimization problem
related to  system (\ref{eq1})-(\ref{eq3}):
\begin{equation}\label{func}
\left\{
\begin{array}{l}
\mbox{Find  }(u,v,f)\in X_2\times X_4\times\mathcal{F}\mbox{ such that the functional }\\
J(u,v,f):=\displaystyle\frac{7\alpha_u}{20}\int_0^T\|u(t)-u_d(t)\|^{20/7}_{L^{20/7}{(\Omega)}}dt+\displaystyle\frac{\alpha_v}{2}\int_0^T\|v(t)-v_d(t)\|^2_{L^2{(\Omega)}}dt
\\
\hspace{2cm}+\displaystyle\frac{\alpha_f}{4}\int_0^T\|f(t)\|^4_{L^4(\Omega_c)}dt\\
\mbox{ is minimized, subject to $(u,v,f)$ satisfies the PDE system (\ref{eq1})-(\ref{eq3}).}
\end{array}
\right.
\end{equation}

Here $(u_d,\, v_d)\in L^{26/7}(Q)\times L^2(Q)$ represent  the desires states (see the beginning of the proof of Theorem \ref{regularity_lagrange} below to justify the regularity
required for 
$u_d\in L^{26/7}(Q)$) and the real numbers $\alpha_u$, $\alpha_v$ and $\alpha_f$ measure the cost of the states and control,
respectively. These numbers satisfy 
$$\alpha_u>0\quad\hbox{and}\quad\alpha_v, \alpha_f\ge0.$$
The  admissible set for the optimal control problem (\ref{func}) is defined by
\begin{equation*}\label{adm}
\mathcal{S}_{ad}=\{s=(u,v,f)\in X_2\times X_4\times\mathcal{F}\,:\, s \mbox{ is a strong solution of  (\ref{eq1})-(\ref{eq3}) in }(0,T)\}.
\end{equation*}
The functional $J$ defined in (\ref{func}) describes the deviation of the cell density $u$ 
and the chemical concentration $v$
from a desired cell density $u_d$ and chemical concentration $v_d$ respectively, plus  the cost of the control measured in the $L^4$-norm.
We also observe that if $(u,v)$ is a weak solution of (\ref{eq1})-(\ref{eq3}) in $(0,T)$ such that $J(u,v,f)<+\infty$, then by Theorem \ref{regularity}, $(u,v)$ is a strong
solution of (\ref{eq1})-(\ref{eq3}) in $(0,T)$. In what follows, we will assume the hypothesis
\begin{equation}\label{hip}
\mathcal{S}_{ad}\ne \emptyset.
\end{equation}
\begin{remark}\label{coment_functional}
The reason for choosing the first term of the objective functional in the $L^{20/7}$-norm is that any weak solution of (\ref{eq1})-(\ref{eq3}) satisfying 
$J(u,v,f)<+\infty$ satisfies that $u \in L^{20/7}(Q)$ and therefore, in virtue of Theorem \ref{regularity}, let us to state that $(u,v)$ is the unique solution of (\ref{eq1})-(\ref{eq3}) 
in the sense of Definition \ref{regul}. Thus, we reduce the admissible states of problem (\ref{func}) to the
strong solutions of (\ref{eq1})-(\ref{eq3}). 
With this formulation we are going to prove the existence of a global  optimal solution and derive the optimality conditions associated to any local optimal solution.
\end{remark}
\subsection{Existence of Global Optimal Solution}
\begin{definition}\label{optimal_solution}
An element $(\tilde{u},\tilde{v},\tilde{f})\in\mathcal{S}_{ad}$ will be called a global optimal solution of problem (\ref{func}) if
\begin{equation}\label{optimal}
J(\tilde{u},\tilde{v},\tilde{f})=\min_{(u,v,f)\in\mathcal{S}_{ad}}J(u,v,f).
\end{equation}

\end{definition}
\begin{theorem}\label{existence_solution}
Let $u_0\in H^1(\Omega)$ and $v_0\in W^{3/2,4}_{{\bf n}}(\Omega)$ with $u_0\ge0$ and $v_0\ge 0$ in $\Omega$. We assume that either $\alpha_f>0$ or $\mathcal{F}$ is bounded in
$L^4(Q_c)$ and hypothesis (\ref{hip}), then the bilinear optimal control problem (\ref{func}) has at least one global optimal solution $(\tilde{u},\tilde{v},\tilde{f})\in\mathcal{S}_{ad}$.

\end{theorem}
\begin{proof}
From  hypothesis (\ref{hip}) $\mathcal{S}_{ad}\neq\emptyset$. Let  $\{s_m\}_{m\in\mathbb{N}}:=\{(u_m,v_m,f_m)\}_{m\in\mathbb{N}}\subset\mathcal{S}_{ad}$ be a minimizing sequence
of $J$, that is, $\displaystyle\lim_{m\rightarrow+\infty}J(s_m)=\inf_{s\in\mathcal{S}_{ad}}J(s)$. Then, by definition of $\mathcal{S}_{ad}$, for each $m\in\mathbb{N}$, $s_m$ satisfies  system
(\ref{eq1}) a.e. $(t,x)\in Q$.

From the definition of $J$ and the assumption $\alpha_f>0$  or  $\mathcal{F}$ is bounded in $L^4(Q_c)$, it follows that 
\begin{equation}\label{bound_F}
\{f_m\}_{m\in\mathbb{N}}\mbox{ is bounded in }L^{4}(Q_c)
\end{equation}
and 
$$
\{u_m\}_{m\in\mathbb{N}}\mbox{ is bounded in }L^{20/7}(Q).
$$

From (\ref{bound-strong}) there exists a positive  constant $C$, independent of $m$, such that
\begin{equation}\label{bound_u_v}
\|u_m,v_m\|_{X_2\times X_4}\le C.
\end{equation}
Therefore, from (\ref{bound_F}), (\ref{bound_u_v}), and taking into account that $\mathcal{F}$ is a closed convex subset of $L^4(Q_c)$
(hence is weakly closed in $L^4(Q_c)$), we deduce that there exists 
$\tilde{s}=(\tilde{u},\tilde{v},\tilde{f})\in X_2\times X_4\times\mathcal{F}$
such that, for some subsequence of $\{s_m\}_{m\in\mathbb{N}}$, still denoted by  $\{s_m\}_{m\in\mathbb{N}}$, 
the following convergences hold, as $m\rightarrow+\infty$:
\begin{eqnarray}
u_m&\rightarrow&\tilde{u}\quad \mbox{weakly in }L^{2}(H^2)\mbox{ and weakly* in }L^\infty(H^1),\label{c2}\\
v_m&\rightarrow&\tilde{v} \quad\mbox{weakly in }L^4(W^{2,4})\mbox{ and weakly* in }L^\infty(W^{3/2,4}_{\bf n}),\label{c3}\\
\partial_tu_m&\rightarrow&\partial_t\tilde{u} \quad\mbox{weakly in } L^2(Q),\label{c4}\\
\partial_tv_m&\rightarrow&\partial_t\tilde{v} \quad\mbox{weakly in } L^4(Q),\label{c5}\\
f_m&\rightarrow&\tilde{f} \quad \mbox{weakly in } L^4(Q_c),\mbox{ and }\tilde{f}\in \mathcal{F}.\label{c6}
\end{eqnarray}
From (\ref{c2})-(\ref{c5}), the Aubin-Lions lemma (see \cite[Th\'eor\`eme 5.1, p. 58]{lions} and \cite[Corollary 4]{simon}) and using Sobolev embedding, we have
\begin{eqnarray}
u_m&\rightarrow&\tilde{u}\quad\mbox{strongly in }C([0,T];L^p)\cap L^2(W^{1,p})\ \ \forall p<6,\label{c7}\\
v_m&\rightarrow&\tilde{v}\quad\mbox{strongly in }C([0,T];L^q)\cap L^4(W^{1,q})\ \ \forall q<+\infty.\label{c8}
\end{eqnarray}
In particular,  we can control the limit of the nonlinear terms of (\ref{eq1}) as follows 
\begin{eqnarray}
\nabla\cdot(u_m\nabla v_m )&\rightarrow& \nabla\cdot(\tilde{u}\nabla\tilde{v}) \quad \hbox{ weakly in $ L^{20/7}(Q)$},\label{c8-0}\\
f_m v_m \chi_{_{\Omega_c}}&\rightarrow& \tilde{f}\,\tilde{v}\,\chi_{_{\Omega_c}} \quad \hbox{ weakly in $ L^{4}(Q)$}.\label{c8-1}
\end{eqnarray}
Moreover, from (\ref{c7}) and (\ref{c8}) we have that $(u_m(0),v_m(0))$ converges to 
$(\tilde{u}(0),\tilde{v}(0))$ in $L^p(\Omega)\times L^q(\Omega)$, and since $u_m(0)=u_0$, $v_m(0)=v_0$, we deduce that $\tilde{u}(0)=u_0$ 
and $\tilde{v}(0)=v_0$. Thus $\tilde{s}$ satisfies the initial conditions given in (\ref{eq2}). 
Therefore, considering the convergences (\ref{c2})-(\ref{c8-1}), we can pass to the limit in (\ref{eq1}) satisfied by $(u_m,v_m,f_m)$, as $m$ goes to $+\infty$,
and we conclude that $\tilde{s}=(\tilde{u},\tilde{v},\tilde{f})$ is also a solution of the system (\ref{eq1}) pointwisely, that is, $\tilde{s}\in\mathcal{S}_{ad}$. Therefore,
\begin{equation}\label{op20}
\lim_{m\rightarrow+\infty}J(s_m)=\inf_{s\in\mathcal{S}_{ad}}J(s)\le J(\tilde{s}).
\end{equation}
On the other hand, since $J$ is lower semicontinuous on $\mathcal{S}_{ad}$, we have $J(\tilde{s})\le \displaystyle\liminf_{m\rightarrow+\infty} J(s_m)$, which jointly to  
(\ref{op20}), implies (\ref{optimal}).
\end{proof}
\subsection{Optimality System Related to Local Optimal Solutions}

We will derive the first-order necessary optimality conditions for
a local optimal solution $(\tilde{u},\tilde{v},\tilde{f})$ of problem (\ref{func}), applying a Lagrange multipliers theorem.
We will base on a generic result given by Zowe et al  \cite{zowe} on the existence of Lagrange multipliers in Banach spaces.
In order to introduce the concepts and results given in \cite{zowe} we consider the following optimization problem
\begin{equation}\label{abs1}
\min_{x\in \mathbb{M}} J(x)\ \mbox{ subject to }G(x)\in \mathcal{N},
\end{equation}
where $J:\mathbb{X}\rightarrow\mathbb{R}$ is a functional, $G:\mathbb{X}\rightarrow \mathbb{Y}$ is an operator,
$\mathbb{X}$ and $\mathbb{Y}$ are Banach spaces,  $\mathbb{M}$ is a nonempty closed convex subset of $\mathbb{X}$ and $\mathcal{N}$ is a nonempty closed convex cone in $\mathbb{Y}$ with vertex at the origin.
The admissible set  for problem (\ref{abs1}) is defined by
$$
\mathcal{S}=\{x\in \mathbb{M}\,:\, G(x)\in \mathcal{N}\}.
$$
For a subset $A$ of $\mathbb{X}$ (or $\mathbb{Y}$), $A^+$ denotes its polar cone, that is
$$
A^+=\{\rho\in \mathbb{X}'\,:\, \langle \rho, a\rangle_{\mathbb{X}'}\ge 0,\ \forall a\in A\}.
$$
\begin{definition}(Lagrange multiplier)\label{abs2}
Let $\tilde{x}\in\mathcal{S}$ be a local  optimal solution for problem (\ref{abs1}). Suppose that $J$ and $G$ are Fr\'echet differentiable in $\tilde{x}$, with derivatives
$J'(\tilde{x})$ and $G'(\tilde{x})$, respectively. Then, any $\xi\in \mathbb{Y}'$ is called a Lagrange multiplier for (\ref{abs1}) at the point $\tilde{x}$ if
\begin{equation}\label{abs3}
\left\{
\begin{array}{l}
\xi\in \mathcal{N}^+,\\
\langle\xi, G(\tilde{x})\rangle_{\mathbb{Y}'}=0,\\
J'(\tilde{x})-\xi\circ G'(\tilde{x})\in \mathcal{C}(\tilde{x})^+,
\end{array}
\right.
\end{equation}
where $\mathcal{C}(\tilde{x})=\{\theta(x-\tilde{x})\,:\, x\in \mathbb{M},\, \theta\ge0\}$ is the conical hull of $\tilde{x}$ in $\mathbb{M}$. 
\end{definition}
\begin{definition}\label{abs4}
Let $\tilde{x}\in\mathcal{S}$ be a local optimal solution for problem (\ref{abs1}). We say that $\tilde{x}$ is a regular point if
\begin{equation*}\label{abs5}
G'(\tilde{x})[\mathcal{C}(\tilde{x})]-\mathcal{N}(G(\tilde{x}))=\mathbb{Y},
\end{equation*}
where $\mathcal{N}(G(\tilde{x}))=\{(\theta(n-G(\tilde{x}))\,:\, n\in \mathcal{N},\, \theta\ge0\}$ is the conical hull of $G(\tilde{x})$ in $\mathcal{N}$.  
\end{definition}
\begin{theorem}(\cite[Theorem 3.1]{zowe})\label{abs6}
Let $\tilde{x}\in\mathcal{S}$ be a local  optimal solution for problem (\ref{abs1}). If $\tilde{x}$ is a regular point, then the set of Lagrange multipliers for (\ref{abs1}) at $\tilde{x}$
is nonempty. 
\end{theorem}
Now, we will reformulate the optimal control problem (\ref{func}) in the abstract setting  (\ref{abs1}).
We consider the following Banach spaces
\begin{equation*}\label{spaces_X_Y}
\mathbb{X}:=\mathcal{W}_u\times\mathcal{W}_v\times L^4(Q_c),\ 
\mathbb{Y}:=L^2(Q)\times L^4(Q)\times H^1(\Omega)\times W^{3/2,4}_{\bf n}(\Omega),
\end{equation*}
where
\begin{eqnarray}
\mathcal{W}_u&:=&\left\{u\in X_2\,:\,\dfrac{\partial u}{\partial{\bf n}}=0\ \mbox{ on }(0,T)\times\partial\Omega\right\},\label{space_state_u}\\
\mathcal{W}_v&:=&\left\{v\in X_4\,:\, \frac{\partial v}{\partial{\bf n}}=0\, \mbox{ on }(0,T)\times\partial\Omega\right\},\label{space_state_v}
\end{eqnarray} 
and the  operator $ {G}=(G_1,G_2,G_3,G_4):\mathbb{X}\rightarrow\mathbb{Y}$, where 
\begin{equation*}
G_1:\mathbb{X}\rightarrow L^2(Q),\ G_2:\mathbb{X}\rightarrow L^4(Q),\ G_3:\mathbb{X}\rightarrow H^1(\Omega),\ G_4:\mathbb{X}\rightarrow W^{3/2,4}_{\bf n}(\Omega)
\end{equation*}
are defined at each point $s=(u,v,f)\in\mathbb{X}$ by
\begin{equation*}\label{restriction}
\left\{
\begin{array}{l}
G_1(s)=\partial_tu-\Delta u-\nabla\cdot(u\nabla v),\vspace{0.1cm}\\
G_2(s)=\partial_tv-\Delta v+v-u-f\,v\,\chi_{_{\Omega_c}},\vspace{0.1cm}\\
G_3(s)=u(0)-u_0,\vspace{0.1cm}\\
G_4(s)=v(0)-v_0.
\end{array}
\right.
\end{equation*}
Thus, the optimal control problem (\ref{func}) is reformulated  as follows
\begin{equation}\label{problem-1}
\min_{s\in{\mathbb{M}}}J(s)\quad \mbox{ subject to }\quad {G}(s)={\bf 0},
\end{equation}
where
\begin{equation*}\label{M}
\mathbb{M}:=\mathcal{W}_u\times\mathcal{W}_v\times\mathcal{F}.
\end{equation*}
and $\mathcal{F}$ is defined in (\ref{F-def}).

We observe that ${\mathbb{M}}$ is a closed convex subset of $\mathbb{X}$, $\mathcal{N}=\{{\bf 0}\}$  and the set of admissible solutions is rewritten as
\begin{equation}\label{admi-1}
\mathcal{S}_{ad}=\{s=(u,v,f)\in{\mathbb{M}}\,:\, {G}(s)={\bf 0}\}.
\end{equation}
Concerning to the differentiability of the constraint operator ${G}$ and the functional $J$ we have the following results.

\begin{lemma}
The functional  $J:\mathbb{X}\rightarrow\mathbb{R}$ is Fr\'echet differentiable and the 
derivative of $J$ in $\tilde{s}=(\tilde{u},\tilde{v},\tilde{f})\in\mathbb{X}$ in the direction
 $r=(U,V,F)\in\mathbb{X}$ is
\begin{eqnarray}\label{C7}
J'(\tilde{s})[r]=\alpha_u\int_0^T\int_{\Omega}{\rm sgn}(\tilde{u}-u_d)|\tilde{u}-u_d|^{13/7}U+\alpha_v\int_0^T\int_{\Omega}(\tilde{v}-v_d)V
+\alpha_f\int_0^T\int_{\Omega_c}(\tilde{f})^3F.
\end{eqnarray}
\end{lemma}
\begin{lemma}
The operator ${G}:\mathbb{X}\rightarrow\mathbb{Y}$  is Fr\'echet differentiable and the  derivative of  ${G}$ in 
$\tilde{s}=(\tilde{u},\tilde{v},\tilde{f})\in\mathbb{X}$ in the direction $r=(U,V,F)\in\mathbb{X}$ is the linear operator

${G}'(\tilde{s})[r]=(G_1'(\tilde{s})[r],G_2'(\tilde{s})[r],G_3'(\tilde{s})[r],G_4'(\tilde{s})[r])$ defined by 
\begin{equation}\label{C8}
\left\{
\begin{array}{rcl}
G_1'(\tilde{s})[r]&=&\partial_tU-\Delta U-\nabla\cdot(U\nabla\tilde{v})-\nabla\cdot(\tilde{u}\nabla V),\\
G_2'(\tilde{s})[r]&=&\partial_tV-\Delta V+V-U-\tilde{f}\,V\,\chi_{_{\Omega_c}}-F\tilde{v},\\
G'_3(\tilde{s})[r]&=&U(0),\\
G'_4(\tilde{s})[r]&=&V(0).
\end{array}
\right.
\end{equation}
\end{lemma}
We wish to prove the existence of Lagrange multipliers,
which is guaranteed if a  local optimal solution of  problem (\ref{problem-1}) is a regular point  of operator $\mathcal{G}$
(in virtue of Theorem \ref{abs6}).
\begin{remark}\label{regular_point}
Since for  problem (\ref{problem-1}) $\mathcal{N}=\{{\bf 0}\}$, then $\mathcal{N}({G}(\tilde{s}))=\{{\bf 0}\}$. Thus, from Definition \ref{abs4}  we conclude that 
$\tilde{s}=(\tilde{u},\tilde{v},\tilde{f})\in\mathcal{S}_{ad}$
is a regular point  if for  any $(g_u,g_v,U_0,V_0)\in\mathbb{Y}$ there exists $r=(U,V,F)\in \mathcal{W}_u\times\mathcal{W}_v\times\mathcal{C}(\tilde{f})$
such that
\begin{equation*}\label{regular}
{G}'(\tilde{s})[r]=(g_u,g_v,U_0,V_0),
\end{equation*}
where $\mathcal{C}(\tilde{f}):=\{\theta(f-\tilde{f})\,:\, \theta\ge0,\, f\in\mathcal{F}\}$ is the conical hull of  $\tilde{f}$ in $\mathcal{F}$. 
\end{remark}
\begin{lemma}\label{regular-lemma}
Let $\tilde{s}=(\tilde{u},\tilde{v},\tilde{f})\in\mathcal{S}_{ad}$  ($\mathcal{S}_{ad}$ defined in (\ref{admi-1})), then $\tilde{s}$ is a regular point.
\end{lemma}

\begin{proof}
Let  $(g_u,g_v,U_0,V_0)\in \mathbb{Y}$. Since $0\in\mathcal{C}(\tilde{f})=\{\theta(f-\tilde{f})\,:\, \theta\ge0,\, f\in\mathcal{F}\}$, it is sufficient to show  the existence of
$(U,V)\in\mathcal{W}_u\times\mathcal{W}_v$ solving the linear problem
\begin{equation}\label{C9}
\left\{
\begin{array}{rcl}
\partial_tU-\Delta U-\nabla\cdot(U\nabla\tilde{v})-\nabla\cdot(\tilde{u}\nabla V)&=&g_u\quad\mbox{ in }Q,\\
\partial_tV-\Delta V+V-U-\tilde{f}\,V\,\chi_{_{\Omega_c}}&=&g_v \quad \mbox{ in }Q,\\
U(0)=U_0,\ V(0)&=&V_0 \quad\mbox{ in }\Omega,\\
\dfrac{\partial U}{\partial{\bf n}}=0,\ \dfrac{\partial V}{\partial{\bf n}}&=&0 \quad\mbox{ on }(0,T)\times\partial\Omega.
\end{array}
\right.
\end{equation}
Since (\ref{C9}) is a linear system we argue in a formal manner, proving that any regular enough solution is bounded in
$\mathcal{W}_u\times\mathcal{W}_v$.
A detailed proof can be made by using, for instance, a Galerkin method.

Testing (\ref{C9})$_1$ by $U$ and (\ref{C9})$_2$ by $-\Delta V$, we have 
\begin{eqnarray}
&&\frac12\frac{d}{dt}(\|U\|^2+\|\nabla V\|^2)+\|\nabla U\|^2+\|\nabla V\|^2+\|\Delta V\|^2\nonumber\\ 
&&\le|(U\nabla\tilde{v},\nabla U)|+|(\tilde{u}\nabla V,\nabla U)|
+|(g_u,U)|+|(U,\Delta V)|+|(\tilde{f}V\,\chi_{_{\Omega_c}},\Delta V)|+|(g_v,\Delta V)|.\label{C10}
\end{eqnarray}
Using the H\"older and Young inequalities  on the terms on the right side of (\ref{C10}) and taking into account
(\ref{interpol}) we obtain 
\begin{eqnarray}
|(U\nabla\tilde{v},\nabla U)|&\le&\|U\|_{L^4}\|\nabla\tilde{v}\|_{L^4}\|\nabla U\|\le C\|U\|^{1/4}\|\nabla\tilde{v}\|_{L^4}\|U\|^{7/4}_{H^1}\nonumber\\
&\le&\delta\|U\|^2_{H^1}+C_\delta\|\nabla\tilde{v}\|^8_{L^4}\|U\|^2,\label{C11}\\
|(\tilde{u}\nabla V,\nabla U)|&\le&\|\tilde{u}\|_{L^4}\|\nabla V\|_{L^4}\|\nabla U\|\le \delta\|\nabla U\|^2+C_\delta\|\tilde{u}\|^2_{L^4}\|\nabla V\|^{1/2}\|\nabla V\|^{3/2}_{H^1}\nonumber\\
&\le&\delta(\|\nabla U\|^2+\|\nabla V\|^2_{H^1})+C_\delta\|\tilde{u}\|^8_{L^4}\|\nabla V\|^2,\label{C12}\\
|(g_u,U)|&\le&\delta\|U\|^2+C_\delta\|g_u\|^2,\label{C13}\\
|(U,\Delta V)|&\le&\delta\|\Delta V\|^2+C_\delta\|U\|^2,\label{C14}\\
|(\tilde{f}\,V\,\chi_{_{\Omega_c}},\Delta V)|&\le&\|\tilde{f}\|_{L^4}\|V\|_{L^4}\|\Delta V\|\le \delta\|\Delta V\|^2+C_\delta\|\tilde{f}\|^2_{L^4}\|V\|^2_{H^1},\label{C15}\\
|(g_v,\Delta V)|&\le&\delta\|\Delta v\|^2+C_\delta\|g_v\|^2.\label{C16}
\end{eqnarray}
On the other hand, testing by $V$ in (\ref{C9})$_2$ we obtain 
\begin{eqnarray}
\frac12\frac{d}{dt}\|V\|^2+\|\nabla V\|^2+\|V\|^2&\le&|(U,V)|+|(\tilde{f}\,V\,\chi_{_{\Omega_c}},V)|+|(g_v,V)|\nonumber\\
&\le&\delta\|V\|^2_{H^1}+C_\delta\|U\|^2+C_\delta\|\tilde{f}\|^2_{L^4}\|V\|^2+C_\delta\|g_v\|^2.\label{C17}
\end{eqnarray}
Summing the inequalities   (\ref{C10}) and  (\ref{C17}), and then adding $\|U\|^2$ to both sides of the inequality obtained, and taking into account (\ref{C11})-(\ref{C16}), for $\delta$ small enough,  we have
\begin{eqnarray}
\frac{d}{dt}(\|U\|^2+\|V\|^2_{H^1})+C\|U\|^2_{H^1}+C\|V\|^2_{H^2}&\le&C(1+\|\nabla\tilde{v}\|^8_{L^4})\|U\|^2+C(\|g_u\|^2+\|g_v\|^2)\nonumber\\
&+&C\|\tilde{u}\|^8_{L^4}\|\nabla V\|^2+C\|\tilde{f}\|^2_{L^4}\|V\|^2_{H^1}.\label{C18}
\end{eqnarray}
From (\ref{C18}) and Gronwall lemma we deduce that there exists a positive constant $C$ that depends on $T$, $\|U_0\|, \|V_0\|_{H^1}, \|\tilde{u}\|_{L^8(L^4)}$,
$\|\nabla\tilde{v}\|_{L^8(L^4)}$, $\|\tilde{f}\|_{L^2(L^4)}$, $\|g_u\|_{L^2(Q)}$ and $\|g_v\|_{L^2(Q)}$ such that
\begin{equation}\label{C19}
\|U,V\|_{L^\infty(L^2\times H^1)\cap L^2(H^1\times H^2)}\le C.
\end{equation}
In particular, from (\ref{C19}) we obtain that  $(U,V)\in L^{{10}/{3}}(Q)\times L^{10}(Q)$, and since  $\tilde{f}\in L^4(Q_c)$ we have
$\tilde{f}\,V\,\chi_{_{\Omega_c}}\in L^{{20}/{7}}(Q)$. Then, applying Lemma \ref{feireisl} (for $p={20}/{7}$)  to (\ref{C9})$_1$, we deduce that
$$
V\in X_{20/7}.
$$
By Sobolev embeddings $V\in L^\infty(Q)$, so that $\tilde{f}\,V\,\chi_{_{\Omega_c}}\in L^4(Q)$. Thus, using that $U\in L^{{10}/{3}}(Q)$, again by Lemma \ref{feireisl}  (for $p={10}/{3}$) we obtain that
\begin{equation}\label{C20}
V\in X_{10/3}.
\end{equation}

Now, testing  (\ref{C9})$_1$ by $-\Delta U$ we have
\begin{eqnarray}\label{C21}
\frac12\frac{d}{dt}\|\nabla U\|^2+\|\Delta U\|^2&\le&|(U\Delta\tilde{v},\Delta U)|+|(\nabla U\cdot\nabla\tilde{v},\Delta U)|
+|(\tilde{u}\Delta V,\Delta U)|\nonumber\\
&+&|(\nabla\tilde{u}\cdot\nabla V,\Delta U)| +|(g_u,\Delta U)|.
\end{eqnarray}
Applying the H\"older and  Young inequalities to the terms on the right side of (\ref{C21}), and using (\ref{interpol}), we have
\begin{eqnarray}
|(U\Delta\tilde{v},\Delta U)|&\le&\|U\|_{L^6}\|\Delta\tilde{v}\|_{L^3}\|\Delta U\|\le C\|U\|_{H^1}\|\Delta\tilde{v}\|_{L^3}\|\Delta U\|\nonumber\\
&\le&\delta\|U\|^2_{H^2}+C_\delta\|U\|^2_{H^1}\|\Delta\tilde{v}\|^2_{L^3},\label{C22}\\
|(\nabla U\cdot\nabla\tilde{v},\Delta U)|&\le&\|\nabla U\|_{L^4}\|\nabla\tilde{v}\|_{L^4}\|\Delta U\|\le C\|\nabla U\|^{1/4}\|\nabla\tilde{v}\|_{L^4}\|U\|^{7/4}_{H^2}\nonumber\\
&\le&\delta\|U\|^2_{H^2}+C_\delta\|\nabla U\|^2\|\nabla\tilde{v}\|^8_{L^4},\label{C23}\\
|(\tilde{u}\Delta V,\Delta U)|
&\le&
\|\tilde{u}\|_{L^6}\|\Delta V\|_{L^3}\|\Delta U\|\le C\|\tilde{u}\|_{H^1}\|\Delta V\|_{L^3}\|\Delta U\|\nonumber\\
&\le&\delta\|U\|^2_{H^2}+C_\delta\|\tilde{u}\|^2_{H^1}\|\Delta V\|^2_{L^3},\label{C24}\\
|(\nabla\tilde{u}\cdot\nabla V,\Delta U)|&\le&\|\nabla\tilde{u}\|_{L^3}\|\nabla V\|_{L^6}\|\Delta U\|\le C\|\nabla\tilde{u}\|_{L^3}\|\nabla V\|_{H^1}\|\Delta U\|\nonumber\\
&\le&\delta\|U\|^2_{H^2}+C_\delta\|\nabla\tilde{u}\|^2_{L^3}\|V\|^2_{W^{7/5,10/3}},\label{C25}\\
|(g_u,\Delta U)|&\le&\delta\|\Delta U\|^2+C_\delta\|g_u\|^2.\label{C26}
\end{eqnarray}
Now, we observe that $\displaystyle\frac{d}{dt}\left(\int_\Omega U\right)=\int_\Omega g_u$, which implies
\begin{equation}\label{C27-1}
\frac12\frac{d}{dt}\left(\int_\Omega U\right)^2=\left(\int_\Omega g_u\right)\left(\int_\Omega U\right)\le C_\delta\left(\int_\Omega g_u\right)^2+\delta\left(\int_\Omega U\right)^2
\end{equation}
and
\begin{equation}\label{C27}
\left|\int_\Omega U(t)\right|^2\le \left|\int_\Omega U_0+\int_0^t\int_\Omega g_u\right|^2\le C.
\end{equation}
Summing inequalities (\ref{C21}), (\ref{C27-1}) and (\ref{C27}), and taking into account (\ref{C22})-(\ref{C26}), for $\delta$ small enough, we obtain 
\begin{eqnarray}\label{C28}
\frac{d}{dt}\|U\|^2_{H^1}+C\|U\|^2_{H^2}&\le&C\|U\|^2_{H^1}\|\Delta\tilde{v}\|^2_{L^3}+C\|\nabla U\|^2\|\nabla\tilde{v}\|^8_{L^4}+C\|\tilde{u}\|^2_{H^1}\|\Delta V\|^2_{L^3}\nonumber\\
&+&C\|\nabla\tilde{u}\|^2_{L^3}\|V\|^2_{W^{7/5,10/3}}+C\|g_u\|^2+C.
\end{eqnarray}

We observe that from (\ref{C20}) we have $ V\in L^{\infty}(W^{7/5,10/3}) \cap L^{10/3}(W^{2,10/3})$, 
and we know that $\tilde{u} \in X_2$, $\tilde{v}\in X_4$.
Then,  from (\ref{C28})  and Gronwall lemma we deduce
\begin{equation*}\label{C29}
U\in L^\infty(H^1)\cap L^2(H^2)\hookrightarrow L^{10}(Q).
\end{equation*}
Now, since $U\in L^{10}(Q)$ and  $\tilde{f}\,V\,\chi_{_{\Omega_c}}\in L^4(Q)$,  we have 
$ U + \tilde{f}\,V\,\chi_{_{\Omega_c}} \in L^4(Q)$. Then, from  (\ref{C9})$_2$ and Lemma~\ref{feireisl} (for $p=4$) we conclude that
$V\in X_4$.

\noindent Finally, using that $(\tilde{u},U)\in L^{10}(Q)^2$, $(\Delta\tilde{v},\Delta V)\in L^4(Q)^2$, $(\nabla \tilde{u},\nabla U)\in L^{10/3}(Q)^2$, and $(\nabla\tilde{v},\nabla V)\in L^{20}(Q)^2$ we deduce
\begin{equation}\label{C30}
\nabla\cdot(U\nabla\tilde{v})+\nabla\cdot(\tilde{u}\nabla V)\in L^{20/7}\hookrightarrow L^2(Q).
\end{equation}
Therefore, thanks to (\ref{C30}), applying Lemma \ref{feireisl}  (for $p=2$) to (\ref{C9})$_1$, we conclude that $U\in X_2$. Thus, the proof is finished.
\end{proof}
\begin{remark}\label{uni-lineal}
Using a classical comparison argument,   inequality (\ref{interpol}) and  Gronwall lemma,   the uniqueness of solutions of system (\ref{C9}) is deduced.
\end{remark}
Now we show the existence of Lagrange multiplier for problem (\ref{func}) associated to any local optimal solution $\tilde{s}=(\tilde{u},\tilde{v},\tilde{f})\in\mathcal{S}_{ad}.$
\begin{theorem}\label{lagrange}
Let $\tilde{s}=(\tilde{u},\tilde{v},\tilde{f})\in\mathcal{S}_{ad}$ be a local optimal solution for the control problem (\ref{func}). Then, there exist a Lagrange multiplier 
$\xi=(\lambda,\eta,\varphi_1,\varphi_2)\in L^2(Q)\times L^{4/3}(Q)\times (H^1(\Omega))'\times (W^{3/2,4}_{\bf n}(\Omega))'$ such that for all 
$(U,V,F)\in \mathcal{W}_u\times\mathcal{W}_v\times\mathcal{C}(\tilde{f})$
\begin{eqnarray}
&&\alpha_u\int_0^T\int_{\Omega}{\rm sgn}(\tilde{u}-u_d)|\tilde{u}-u_d|^{13/7}U+\alpha_v\int_0^T\int_{\Omega}(\tilde{v}-v_d)V+\alpha_f\int_0^T\int_{\Omega_c}(\tilde{f})^3F\nonumber\\
&&-\int_0^T\int_\Omega\bigg(\partial_tU-\Delta U-\nabla\cdot(U\nabla\tilde{v})-\nabla\cdot(\tilde{u}\nabla V)\bigg)\lambda
-\int_0^T\int_\Omega\bigg(\partial_tV-\Delta V+V-U-\tilde{f}V\chi_{_{\Omega_c}}\bigg)\eta\nonumber\\
&&-\int_\Omega U(0)\varphi_1-\int_\Omega V(0)\varphi_2+\int_0^T\int_{\Omega_c}F\tilde{v}\eta\ge0.\label{M1}
\end{eqnarray}
\end{theorem}
\begin{proof}
From Lemma \ref{regular-lemma}, $\tilde{s}\in\mathcal{S}_{ad}$ is a regular point, then from Theorem \ref{abs6} 
there exists a Lagrange multiplier  
$\xi=(\lambda,\eta,\varphi_1,\varphi_2)\in L^2(Q)\times L^{4/3}(Q)\times (H^1(\Omega))'\times (W^{3/2,4}_{\bf n}(\Omega))'$ such that
by (\ref{abs3})$_3$ one must satisfy
\begin{equation}\label{M1-1}
J'(\tilde{s})[r]-\langle R_1'(\tilde{s})[r],\lambda\rangle-\langle R_2'(\tilde{s})[r],\eta\rangle-\langle R_3'(\tilde{s})[r],\varphi_1\rangle
-\langle R_4'(\tilde{s})[r],\varphi_2\rangle\ge 0,
\end{equation}
for all $r=(U,V,F)\in\mathcal{W}_u\times\mathcal{W}_v\times\mathcal{C}(\tilde{f}).$ Thus, the proof follows from (\ref{C7}), (\ref{C8}) and (\ref{M1-1}).
\end{proof}

From Theorem \ref{lagrange}, we derive an optimality system for problem (\ref{func}), by considering the  spaces 
\begin{equation*}\label{spaces_optimality}
\mathcal{W}_{u_0}=\{u\in\mathcal{W}_u\,:\, u(0)=0\},\quad \mathcal{W}_{v_0}=\{v\in\mathcal{W}_v\,:\, v(0)=0\}.
\end{equation*}
\begin{corol}
Let $\tilde{s}=(\tilde{u},\tilde{v},\tilde{f})\in\mathcal{S}_{ad}$ be a local optimal solution for the control problem (\ref{func}). Then the Lagrange multiplier
$(\lambda,\eta)\in L^2(Q)\times L^{4/3}(Q)$, provided by Theorem \ref{lagrange}, satisfies the system
\begin{eqnarray}
&&\int_0^T\int_\Omega\bigg(\partial_tU-\Delta U-\nabla\cdot(U\nabla\tilde{v})\bigg)\lambda-\int_0^T\int_\Omega U\eta\nonumber\\
&&\hspace{0.7cm}=\alpha_u\int_0^T\int_{\Omega}{\rm sgn}(\tilde{u}-u_d)|\tilde{u}-u_d|^{13/7}U,\qquad  \forall \, U\in\mathcal{W}_{u_0}\label{M2},\\
&&\int_0^T\int_\Omega\bigg(\partial_t V-\Delta V+V\bigg)\eta-\int_0^T\int_{\Omega_c}\tilde{f}V\eta-\int_0^T\int_\Omega\nabla\cdot(\tilde{u}\nabla V)\lambda\nonumber\\
&&\hspace{0.7cm}=\alpha_v\int_0^T\int_{\Omega}(\tilde{v}-v_d)V,\qquad  \forall \,V\in\mathcal{W}_{v_0}\label{M3},
\end{eqnarray}

and the optimality condition
\begin{equation}\label{M5}
\int_0^T\int_{\Omega_c}(\alpha_f(\tilde{f})^3+\tilde{v}\eta)(f-\tilde{f})\ge 0\qquad \forall f\in\mathcal{F}.
\end{equation}
\end{corol}
\begin{proof}
From (\ref{M1}), taking  $(V,F)=(0,0)$, and using that $\mathcal{W}_{u_0}$ is a vectorial space, we have (\ref{M2}). Similarly, taking $(U,F)=(0,0)$ in (\ref{M1}), and taking into account
that $\mathcal{W}_{v_0}$ is a vectorial space, we deduce (\ref{M3}).
Finally, taking $(U,V)=(0,0)$ in (\ref{M1}) we have
$$
\alpha_f\int_0^T\int_{\Omega_c}(\tilde{f})^3F+\int_0^T\int_{\Omega_c}\tilde{v}\eta F\ge0\quad \forall \, F\in\mathcal{C}(\tilde{f}).
$$
Thus, choosing $F=\theta(f-\tilde{f})\in \mathcal{C}(\tilde{f})$ for all $f\in\mathcal{F}$ and $\theta\ge0$ in the last inequality, we have (\ref{M5}). 
\end{proof}
\begin{remark}\label{very-weak}
A pair $(\lambda,\eta)\in L^2(Q)\times L^{4/3}(Q)$ satisfying (\ref{M2})-(\ref{M3}) corresponds to the concept of very weak solution of the linear system
\begin{equation}\label{M4}
\left\{
\begin{array}{rcl}
\partial_t\lambda+\Delta\lambda-\nabla\lambda\cdot\nabla\tilde{v}+\eta&=&-\alpha_u{\rm sgn}(\tilde{u}-u_d)|\tilde{u}-u_d|^{13/7}\quad \mbox{ in }Q,\\
\partial_t\eta+\Delta\eta+\nabla\cdot(\tilde{u}\nabla\lambda)-\eta+\tilde{f}\,\eta\,\chi_{_{\Omega_c}}&=&-\alpha_v(\tilde{v}-v_d)\quad \mbox{ in }Q,\\
\lambda(T)=0,\ \eta(T)&=&0\quad \mbox{ in }\Omega,\\
\dfrac{\partial\lambda}{\partial{\bf n}}=0,\ \dfrac{\partial\eta}{\partial{\bf n}}&=&0\quad \mbox{ on }(0,T)\times\partial\Omega.
\end{array}
\right.
\end{equation}
\end{remark}

\begin{theorem}\label{regularity_lagrange}
Let $\tilde{s}=(\tilde{u},\tilde{v},\tilde{f})\in\mathcal{S}_{ad}$ be a local optimal solution for the problem (\ref{func}) and $u_d\in L^{26/7}(Q)$. Then
the system (\ref{M4}) has a unique solution $(\lambda,\eta)$ such that
\begin{eqnarray}
&&\lambda\in X_2,\label{reg1}\\
&&\eta\in X_{5/3}.\label{regg2}
\end{eqnarray}
\end{theorem}
\begin{proof}
Since the desired state $u_d\in L^{26/7}(Q)$, we have that $h(\tilde{u}):={\rm sgn}(\tilde{u}-u_d)|\tilde{u}-u_d|^{13/7}\in L^2(Q)$. 
In fact, $\tilde{u}$ is more regular because 
assuming $\tilde{u}\in L^{20/7}(Q)$, it can be proved that
$\tilde{u}\in L^\infty(H^1)\cap L^2(H^2)\hookrightarrow L^{10}(Q)$ (see the proof of the Theorem \ref{regularity} for more details).

Let $s=T-t$, with $t\in(0,T)$
and $\tilde{\lambda}(s)=\lambda(t)$, $\tilde{\eta}(s)=\eta(t)$. Then, system (\ref{M4}) is equivalent to 
\begin{equation}\label{R1}
\left\{
\begin{array}{rcl}
\partial_s\tilde{\lambda}-\Delta\tilde{\lambda}+\nabla\tilde{\lambda}\cdot\nabla\tilde{v}-\tilde{\eta}&=&\alpha_uh(\tilde{u})\quad \mbox{ in }Q,\\
\partial_s\tilde{\eta}-\Delta\tilde{\eta}-\nabla\cdot(\tilde{u}\nabla\tilde{\lambda})+\tilde{\eta}-\tilde{f}\,\tilde{\eta}\,\chi_{_{\Omega_c}}&=&\alpha_v(\tilde{v}-v_d)\quad \mbox{ in }Q,\\
\tilde{\lambda}(0)=0,\ \tilde{\eta}(0)&=&0\quad \mbox{ in }\Omega,\\
\dfrac{\partial\tilde{\lambda}}{\partial{\bf n}}=0,\ \dfrac{\partial\tilde{\eta}}{\partial{\bf n}}&=&0\quad \mbox{ on }(0,T)\times\partial\Omega.
\end{array}
\right.
\end{equation}
Testing (\ref{R1})$_1$ by $-\Delta\tilde{\lambda}$  and (\ref{R1})$_2$ by $\tilde{\eta}$, and using H\"older and Young inequalities,  we can obtain 
\begin{eqnarray}\label{R2}
\frac12\frac{d}{ds}(\|\nabla\tilde{\lambda}\|^2+\|\tilde{\eta}\|^2)+\|\Delta\tilde{\lambda}\|^2+\|\tilde{\eta}\|^2_{H^1}\le
\delta(\|\nabla\tilde{\lambda}\|^2_{H^1}+\|\Delta\tilde{\lambda}\|^2+\|\nabla\tilde{\eta}\|^2)
+C_\delta(1+\|\tilde{f}\|^{8/5}_{L^4})\|\tilde{\eta}\|^2\nonumber\\
\hspace*{-0.6cm}+C_\delta(\|\tilde{u}\|^8_{L^4}+\|\nabla\tilde{v}\|^8_{L^4})\|\nabla\tilde{\lambda}\|^2+C_\delta(\|h(\tilde{u})\|^2+\|\tilde{v}-v_d\|^2).
\end{eqnarray}
Now, since $\dfrac{\partial\tilde{\lambda}}{\partial{\bf n}}=0$ on $\partial\Omega$,  then by \cite[Corollary 3.5]{amrouche} we have
\begin{equation}\label{R2-1}
\|\nabla\tilde{\lambda}\|^2_{H^1}\simeq \|\nabla\tilde{\lambda}\|^2+\|\Delta\tilde{\lambda}\|^2.
\end{equation}
Thus, taking $\delta$ small enough, from (\ref{R2}) and (\ref{R2-1}) we deduce the following energy inequality 
\begin{eqnarray*}
\frac{d}{ds}(\|\nabla\tilde{\lambda}\|^2+\|\tilde{\eta}\|^2)
+C(\|\nabla \tilde{\lambda}\|_{H^1}^2+\|\tilde{\eta}\|^2_{H^1})&\le&
C(\|\tilde{u}\|^8_{L^4}+\|\nabla\tilde{v}\|^8_{L^4}+1)\|\nabla\tilde{\lambda}\|^2+C(1+\|\tilde{f}\|^{8/5}_{L^4})\|\tilde{\eta}\|^2\\
&+&C(\|h(\tilde{u})\|^2+\|\tilde{v}-v_d\|^2),
\end{eqnarray*}
which, jointly with Gronwall lemma, implies 
$$
(\nabla\tilde{\lambda}, \tilde{\eta})\in L^\infty(L^2)\cap L^2(H^1)\hookrightarrow L^{10/3}(Q).
$$
In particular, using that $(\nabla\tilde{\lambda},\nabla\tilde{v})\in L^{10/3}(Q)\times L^{20}(Q)$, we have $\nabla\tilde{\lambda}\cdot\nabla\tilde{v}\in L^{20/7}(Q)\hookrightarrow L^2(Q)$. Thus,
applying Lemma \ref{feireisl} (for $p=2$) to  (\ref{R1})$_1$, we deduce  (\ref{reg1}).

On the other hand, since $\tilde{f}\in L^4(Q_c)$, $\tilde{\eta}\in L^{{10}/{3}}(Q)$, we have
\begin{equation}\label{R3}
\tilde{f}\,\tilde{\eta}\,\chi_{_{\Omega_c}}\in L^{{20}/{11}}(Q).
\end{equation}
Now, taking into account that $\tilde{u}\in L^\infty(H^1)\cap L^2(H^2)\hookrightarrow L^{10}(Q),$ $\Delta\tilde{\lambda}\in L^2(Q)$, and 
$\nabla\tilde{u},\nabla\tilde{\lambda}\in L^{{10}/{3}}(Q)$, we deduce 
\begin{equation}\label{R4}
\nabla\cdot(\tilde{u}\nabla\tilde{\lambda})=\tilde{u}\Delta\tilde{\lambda}+\nabla\tilde{u}\cdot\nabla\tilde{\lambda}\in L^{5/3}(Q).
\end{equation}
Therefore, from (\ref{R1})$_2$, (\ref{R3}), (\ref{R4}) and Lemma \ref{feireisl}  (for $p=5/3$) we obtain (\ref{regg2}).
\end{proof}
In the following result, we obtain more regularity for the  Lagrange multiplier $(\lambda,\eta)$ than provided by Theorem \ref{lagrange}.
\begin{theorem}\label{regularity-multiplier}
Let $\tilde{s}=(\tilde{u},\tilde{v},\tilde{f})\in\mathcal{S}_{ad}$ be a local optimal solution for the control problem (\ref{func}). Then the Lagrange
multiplier, provided by Theorem \ref{lagrange}, satisfies $(\lambda,\eta)\in X_2\times X_{5/3}$.
\end{theorem}
\begin{proof}
Let $(\lambda,\eta)$ be the Lagrange multiplier given in Theorem \ref{lagrange}, which is a very weak solution of problem (\ref{M4}). 
In particular, $(\lambda,\eta)$ satisfies (\ref{M2})-(\ref{M3}).

On the other hand, from Theorem \ref{regularity_lagrange}, system (\ref{M4}) has a unique solution $(\overline{\lambda},\overline{\eta})\in X_2\times X_{5/3}$. Then, it suffices to identify
$(\lambda,\eta)$ with $(\overline{\lambda},\overline{\eta})$. With this objective, we consider the unique solution $(U,V)\in \mathcal{W}_{u}\times\mathcal{W}_v$ of linear system 
(\ref{C9}) for $g_u:=\lambda-\overline{\lambda}\in L^2(Q)$ and $g_v:={\rm sgn}(\eta-\overline{\eta})|\eta-\overline{\eta}|^{1/3}\in L^4(Q)$ (see Lemma \ref{regular-lemma} and Remark \ref{uni-lineal}). 
Then, written (\ref{M4}) for $(\overline{\lambda},\overline{\eta})$ (instead of $(\lambda,\eta)$), testing the first equation by $U$, and the second one by $V$, and  integrating by parts in $\Omega$, we obtain
\begin{equation}\label{RL2}
\int_0^T\int_\Omega\bigg(\partial_tU-\Delta U-\nabla\cdot(U\nabla\tilde{v})\bigg)\overline{\lambda}-\int_0^T\int_\Omega U\overline{\eta}
=\alpha_u\int_0^T\int_{\Omega}{\rm sgn}(\tilde{u}-u_d)|\tilde{u}-u_d|^{13/7}U,
\end{equation}
\begin{equation}\label{RL3}
\int_0^T\int_\Omega\bigg(\partial_tV-\Delta V+V-\tilde{f}V\chi{_{_{\Omega_c}}}\bigg)\overline{\eta}-\int_0^T\int_{\Omega}\nabla\cdot(\tilde{u}\nabla V)\overline{\lambda}
=\alpha_v\int_0^T\int_{\Omega}(\tilde{v}-v_d)V.
\end{equation}
Making the difference between (\ref{M2}) for $(\lambda,\eta)$ and (\ref{RL2}) for $(\overline{\lambda},\overline{\eta})$, and between (\ref{M3}) and (\ref{RL3}), and then adding the respective 
equations, since the right-hand side terms vanish, we have
\begin{eqnarray}\label{RL4}
\int_0^T\int_\Omega\bigg(\partial_tU-\Delta U-\nabla\cdot(U\nabla\tilde{v})-\nabla\cdot(\tilde{u}\nabla V)\bigg)(\lambda-\overline{\lambda})\nonumber\\
+\int_0^T\int_\Omega\bigg(\partial_tV-\Delta V+V-U-\tilde{f}V\chi_{_{\Omega_c}}\bigg)(\eta-\overline{\eta})=0.
\end{eqnarray}
Therefore, taking into account that $(U,V)$ is the unique solution of (\ref{C9}) for $g_u=\lambda-\overline{\lambda}$ and $g_v={\rm sgn}(\eta-\overline{\eta})|\eta-\overline{\eta}|^{1/3}$, from (\ref{RL4}) we deduce 
$$
\|\lambda-\overline{\lambda}\|^2_{L^2(Q)}+\|\eta-\overline{\eta}\|^{4/3}_{L^{4/3}(Q)}=0,
$$
which implies that $(\lambda,\eta)=(\overline{\lambda},\overline{\eta})$ in $L^2(Q)\times L^{4/3}(Q)$. As a consequence of the regularity of $(\overline{\lambda},\overline{\eta})$ we deduce that
$(\lambda,\eta)\in X_2\times X_{5/3}$.
\end{proof}

\begin{corol}(Optimality System)
Let $\tilde{s}=(\tilde{u},\tilde{v},\tilde{f})\in\mathcal{S}_{ad}$ be a local optimal solution for the control problem (\ref{func}). Then, the Lagrange multiplier
$(\lambda,\eta)\in X_2\times X_{5/3}$ satisfies the  optimality system
\begin{equation}\label{R5}
\left\{
\begin{array}{rcl}
\partial_t\lambda+\Delta\lambda-\nabla\lambda\cdot\nabla\tilde{v}+\eta&=&-\alpha_u{\rm sgn}(\tilde{u}-u_d)|\tilde{u}-u_d|^{13/7}\ \ \mbox{ a.e. }(t,x)\in Q,\\
\partial_t\eta+\Delta\eta+\nabla\cdot(\tilde{u}\nabla\lambda)-\eta+\tilde{f}\,\eta\,\chi_{_{\Omega_c}}&=&-\alpha_v(\tilde{v}-v_d)\ \ \mbox{ a.e. }(t,x)\in Q,\\
\lambda(T)=0,\ \eta(T)&=&0\quad  \mbox{ in }\Omega,\\
\dfrac{\partial\lambda}{\partial{\bf n}}=0,\ \dfrac{\partial\eta}{\partial{\bf n}}&=&0\quad\mbox{ on }(0,T)\times\partial\Omega,\\
\displaystyle\int_0^T\int_{\Omega_c}(\alpha_f(\tilde{f})^3+\tilde{v}\,\eta)(f-\tilde{f})&\ge&0\ \ \ \forall f\in\mathcal{F}.
\end{array}
\right.
\end{equation}
\end{corol}
\begin{remark}
If there is no convexity constraint on the control, that is, $\mathcal{F}\equiv L^4(Q_c)$, then (\ref{R5})$_5$ becomes 
$$
\alpha_f(\tilde{f})^3\chi_{_{\Omega_c}}+\tilde{v}\,\eta\,\chi_{_{\Omega_c}}=0.
$$
Thus, the control $\tilde{f}$ is given by 
$$
\tilde{f}=\left(-\frac{1}{\alpha_f}\tilde{v}\,\eta\right)^{1/3}\chi_{_{\Omega_c}}.
$$
\end{remark}
\addcontentsline{toc}{section}{Appendix: Existence of Strong Solutions of Problem (\ref{regg-1})}

\section*{Appendix: Existence of Strong Solutions of Problem (\ref{regg-1})}
\label{appendix}

In this appendix we will prove  Theorem \ref{teo_reg}.

Let us  introduce the {\it weak} space
\begin{equation*}\label{reg3}
\mathcal{X}:= L^\infty(L^2)\cap L^2(H^1).
\end{equation*}

We define the operator $R:\mathcal{X}\times\mathcal{X}\rightarrow X_{5/3}\times X_{10/3}\hookrightarrow\mathcal{X}\times\mathcal{X}$
by $R(\overline{u}^\varepsilon,\overline{z}^\varepsilon)=(u^\varepsilon,z^\varepsilon)$ the solution of the decoupled linear problem

\begin{equation}\label{regg-2}
\left\{
\begin{array}{rcl}
\partial_tu^\varepsilon-\Delta u^\varepsilon&=&\nabla\cdot(\overline{u}_+^\varepsilon\nabla v(\overline{z}^\varepsilon))\quad \mbox{ in }Q,\\
\partial_tz^\varepsilon-\Delta z^\varepsilon&=&\overline{u}^\varepsilon+f\,v(\overline{z}^\varepsilon)_+\chi_{_{\Omega_c}}-\overline{z}^\varepsilon\quad \mbox{ in }Q,\\
u^\varepsilon(0)=u_0^\varepsilon,\ z^\varepsilon(0)&=&v_0^\varepsilon-\varepsilon\Delta v_0^\varepsilon\quad \mbox{ in }\Omega,\\
\dfrac{\partial u^\varepsilon}{\partial {\bf n}}=0,\ \dfrac{\partial z^\varepsilon}{\partial{\bf n}}&=&0 \quad \mbox{ on }(0,T)\times\partial\Omega,
\end{array}
\right.
\end{equation}
where $\overline{v}^\varepsilon:=v(\overline{z}^\varepsilon)$ is the unique solution of problem (\ref{regg-1-1}). In this Appendix, we will denote 
$v(\overline{z}^\varepsilon)$ only by $\overline{v}^\varepsilon$. Then, a solution of system (\ref{regg-1}) is a fixed point of
$R$. Therefore, in order to prove the existence of  solution to  system (\ref{regg-1}) we will use the Leray-Schauder fixed point theorem. In the following lemmas, we will prove the hypotheses 
of such fixed point theorem.

\begin{lemma}\label{compact}
The operator $R:\mathcal{X}\times\mathcal{X}\rightarrow \mathcal{X}\times\mathcal{X}$ is well defined and compact.
\end{lemma}
\begin{proof}
Let $(\overline{u}^\varepsilon,\overline{z}^\varepsilon)\in\mathcal{X}\times\mathcal{X}$. Then, from the $H^2$ and $H^3$-regularity of problem (\ref{regg-1-1}) 
(see \cite[Theorem 2.4.2.7 and Theorem 2.5.11]{grisvard} respectively) we have $\overline{v}^\varepsilon\in L^\infty(H^2)\cap L^2(H^3)$.
Thus, we conclude that $\nabla\overline{v}^\varepsilon\in L^\infty(H^1)\cap L^2(H^2)\hookrightarrow L^{10}(Q)$, and taking into account 
that $(\overline{u}^\varepsilon,\overline{z}^\varepsilon)\in\mathcal{X}\times\mathcal{X}$, we have 
$\nabla\cdot(\overline{u}^\varepsilon_+\nabla \overline{v}^\varepsilon)=\overline{u}^\varepsilon_+\Delta \overline{v}^\varepsilon+\nabla\overline{u}^\varepsilon_+\cdot\nabla \overline{v}^\varepsilon\in L^{{5}/{3}}(Q)$. 
Then, by Lemma \ref{feireisl} (for $p=5/3$),  there exists a unique  solution $u^\varepsilon \in X_{5/3}$ of (\ref{regg-2})$_1$   such that
\begin{equation}\label{com-1}
\|u^\varepsilon\|_{X_{5/3}}\le C(\|u^\varepsilon_0\|_{W^{4/5,5/3}},\|\overline{u}^\varepsilon\|_{\mathcal{X}},\|\overline{z}^\varepsilon\|_{\mathcal{X}}).
\end{equation}

Now, since $\mathcal{X}\hookrightarrow L^{10/3}(Q)$ and $\overline{v}^\varepsilon\in L^\infty(Q)$, we have  
$\overline{u}^\varepsilon+f\,\overline{v}^\varepsilon_+\chi_{_{\Omega_c}}-\overline{z}^\varepsilon\in L^{10/3}(Q)$. Then, by Lemma \ref{feireisl}
(for $p={10}/{3}$), there exists a unique solution $z^\varepsilon$ of (\ref{regg-2})$_2$ belonging to $X_{10/3}$  such that
\begin{equation}\label{com-2}
\|z^\varepsilon\|_{X_{10/3}}\le C(\|z_0^\varepsilon\|_{W_{\bf n}^{{7}/{5},{10}/{3}}},\|\overline{u}^\varepsilon\|_{\mathcal{X}},\|\overline{z}^\varepsilon\|_{\mathcal{X}},\|f\|_{L^4(Q)}).
\end{equation}

Therefore, $R$ is well defined. The compactness of $R$ is consequence of estimates (\ref{com-1}) and (\ref{com-2}), and the compact embedding 
$X_{5/3}\times X_{10/3}\hookrightarrow\mathcal{X}\times\mathcal{X}$. 
 Indeed,  it suffices to prove only the compact embedding $X_{5/3}\hookrightarrow\mathcal{X}$, because $X_{10/3}\hookrightarrow X_{5/3}$.
Let $u\in X_{5/3}$,  then from  Lemma \ref{l6}  we have $W^{4/5,5/3}(\Omega)\hookrightarrow H^{1/2}(\Omega)$ and $W^{2,5/3}(\Omega)\hookrightarrow H^{17/10}(\Omega)$; thus
\begin{equation}\label{compact-1}
u\in X_{5/3}\hookrightarrow L^\infty(H^{1/2})\cap L^{5/3}(H^{17/10}).
\end{equation}
 Then, from (\ref{compact-1}) and Lemma \ref{l7} (for $(p_1,s_1)=(\infty,1/2)$ and $(p_2,s_2)=(5/3,17/10)$) we deduce that
\begin{equation}\label{compact-2}
u\in L^\infty(H^{1/2})\cap L^{5/3}(H^{17/10})\hookrightarrow L^2(H^{3/2}).
\end{equation}
Therefore, since the embedding $H^{3/2}(\Omega)\hookrightarrow H^1(\Omega)$ is compact and $\partial_tu\in L^{5/3}(Q)$, from
\cite[Th\'eor\`eme 5.1, p. 58]{lions} and (\ref{compact-2}) we obtain that $X_{5/3}$ is compactly embedded in $\mathcal{X}$.
\end{proof}
\begin{lemma}\label{fix}
Let $(u^\varepsilon_0,v^\varepsilon_0-\varepsilon\Delta v_0^\varepsilon)\in W^{4/5,5/3}(\Omega)\times W^{7/5,10/3}_{{\bf n}}(\Omega)$ with $u_0^\varepsilon\ge 0$ in $\Omega$ 
and $f\in L^4(Q_c)$. Then, the fixed points of $\alpha R$ are bounded in $\mathcal{X}\times\mathcal{X}$, independently of $\alpha\in[0,1]$, with $u^\varepsilon\ge0$.
\end{lemma}

\begin{proof}
We assume $\alpha\in (0,1]$. Notice that if $(u^\varepsilon,z^\varepsilon)$ is a fixed point of $\alpha R(u^\varepsilon,z^\varepsilon)$, then $(u^\varepsilon,z^\varepsilon)$ satisfies 
\begin{equation}\label{F-1}
\left\{
\begin{array}{rcl}
\partial_tu^\varepsilon-\Delta u^\varepsilon&=&\alpha\,\nabla\cdot({u}_+^\varepsilon\nabla {v}^\varepsilon)\quad \mbox{ in }Q,\\
\partial_tz^\varepsilon-\Delta z^\varepsilon&=&\alpha\,{u}^\varepsilon+\alpha\, f\,{v}^\varepsilon_+\chi_{_{\Omega_c}}-\alpha\,{z}^\varepsilon\quad \mbox{ in }Q,\\
u^\varepsilon(0)=u_0^\varepsilon,\ z^\varepsilon(0)&=&v_0^\varepsilon-\varepsilon\Delta v_0^\varepsilon\quad \mbox{ in }\Omega,\\
\dfrac{\partial u^\varepsilon}{\partial {\bf n}}=0,\ \dfrac{\partial z^\varepsilon}{\partial{\bf n}}&=&0\quad \mbox{ on }(0,T)\times\partial\Omega.
\end{array}
\right.
\end{equation}
The proof is carried out in three  steps:

\vspace{0.2cm}
\underline{Step 1:} $u^\varepsilon\ge 0$ and $\displaystyle\int_\Omega u(t)={m}^\varepsilon_0.$
\vspace{0.2cm}

\noindent Let $(u^\varepsilon,v^\varepsilon)$ be a solution of (\ref{F-1}), then $\partial_tu^\varepsilon$, $\Delta v^\varepsilon$ and $\nabla\cdot(u^\varepsilon_+\nabla v^\varepsilon)$ 
belong to $L^{5/3}(Q)$. Testing (\ref{F-1})$_1$  by $u^\varepsilon_-\in \mathcal{X}\hookrightarrow L^{{10}/{3}}(Q)\hookrightarrow L^{5/2}(Q)$, where $u^\varepsilon_-:=\min\{u^\varepsilon,0\}\le 0$, 
and taking into account that $u^\varepsilon_-=0$ if $u^\varepsilon\ge 0$; $\nabla u^\varepsilon_-=\nabla u^\varepsilon$ if $u^\varepsilon\le 0$,
and $\nabla u^\varepsilon_-=0$ if $u^\varepsilon>0$, we have 
$$
\frac12\frac{d}{dt}\|u^\varepsilon_-\|^2+\|\nabla u^\varepsilon_-\|^2=-\alpha(u_+^\varepsilon\nabla v^\varepsilon,\nabla u_-^\varepsilon)=0,
$$
which implies that $u^\varepsilon_-\equiv 0$ and, consequently, $u^\varepsilon\ge 0$ and, therefore, $u^\varepsilon_+=u^\varepsilon$.
Finally, integrating (\ref{F-1})$_1$ in $\Omega$
and using (\ref{reg1-2})$_1$ we obtain $\displaystyle\int_\Omega u^\varepsilon(t)={m}^\varepsilon_0$.
\vspace{0.2cm}

\underline{Step 2:} $z^\varepsilon$ is bounded in $\mathcal{X}$.
\vspace{0.2cm}

\noindent We observe that $u^\varepsilon+1\ge 1$ and  $u^\varepsilon+1 \in L^{\infty}(L^1)$. Then, in particular, $u^\varepsilon+1\in L^1(Q)$ and 
$$
\frac{2}{5}\ln(u^\varepsilon+1)=\ln(u^\varepsilon+1)^{2/{5}}\le (u^\varepsilon+1)^{{2}/{5}}\in L^{5/2}(Q),
$$
hence $\ln(u^\varepsilon+1)\in L^{5/2}(Q)$. 

Now, testing (\ref{F-1})$_1$ by $\ln(u^\varepsilon+1)\in L^{5/2}(Q)$ and (\ref{F-1})$_2$ by $-\Delta v^\varepsilon\in L^{10/3}(W^{2,10/3})$  (rewritten in terms of $v^\varepsilon$) we have
\begin{eqnarray}
&&\frac{d}{dt}\left(\int_\Omega(u^\varepsilon+1)\ln(u^\varepsilon+1)+\frac12\|\nabla v^\varepsilon\|^2+\frac\varepsilon2\|\Delta v^\varepsilon\|^2\right)
+4\|\nabla\sqrt{u^\varepsilon+1}\|^2\nonumber\\
&&\hspace{0.3cm}+\|\Delta v^\varepsilon\|^2+\alpha\|\nabla v^\varepsilon\|^2+\alpha\varepsilon\|\Delta v^\varepsilon\|^2+\varepsilon\|\nabla(\Delta v^\varepsilon)\|^2\nonumber\\
&&= -\alpha\int_\Omega\frac{u^\varepsilon}{u^\varepsilon+1}\nabla v^\varepsilon\cdot\nabla u^\varepsilon+\alpha\int_\Omega\nabla u^\varepsilon\cdot\nabla v^\varepsilon
-\alpha\int_\Omega f\,v^\varepsilon_+\chi_{_{\Omega_c}} \Delta v^\varepsilon\nonumber\\
&&=\alpha\int_\Omega\frac{1}{u^\varepsilon+1}\nabla u^\varepsilon\cdot\nabla v^\varepsilon-\alpha\int_\Omega f\,v^\varepsilon_+\chi_{_{\Omega_c}}\Delta v^\varepsilon.\label{F-2}
\end{eqnarray}
Applying  H\"older and Young inequalities, we have 
\begin{eqnarray}
\alpha\int_\Omega\frac{1}{u^\varepsilon+1}\nabla u^\varepsilon\cdot\nabla v^\varepsilon
&\le&\frac\alpha2\int_\Omega\frac{|\nabla u^\varepsilon|^2}{u^\varepsilon+1}+\frac{\alpha}{2}\int_\Omega\frac{|\nabla v^\varepsilon|^2}{u^\varepsilon+1}
\le2\alpha\|\nabla\sqrt{u^\varepsilon+1}\|^2+\frac\alpha2\|\nabla v^\varepsilon\|^2,\label{F-3}\\
-\alpha\int_\Omega f\,v^\varepsilon_+\chi_{_{\Omega_c}} \Delta v^\varepsilon
&\le&\alpha\|f\|_{L^4}\|v^\varepsilon\|_{L^4}\|\Delta v^\varepsilon\|\le\delta\|v^\varepsilon\|^2_{H^2}+\alpha^2C_\delta\|f\|_{L^4}^{2}\|v^\varepsilon\|^2_{H^1}.\label{F-4}
\end{eqnarray}

Moreover, integrating (\ref{F-1})$_2$ in $\Omega$, using (\ref{reg1-2}), and taking into account that $v^\varepsilon$ is the unique solution of the problem (\ref{regg-1-1}),
we have 
$$
\frac{d}{dt}\left(\int_\Omega v^\varepsilon\right)+\int_\Omega v^\varepsilon=\alpha\,{m}^\varepsilon_0+\alpha\int_\Omega f\,v^\varepsilon_+\chi_{_{\Omega_c}}.
$$
Multiplying this equation by $\displaystyle\int_\Omega v^\varepsilon$ and using the H\"older and Young inequalities we obtain
\begin{eqnarray}
\frac12\frac{d}{dt}\left(\int_\Omega v^\varepsilon\right)^2+\left(\int_\Omega v^\varepsilon\right)^2&=&\alpha\, {m}^\varepsilon_0\left(\int_\Omega v^\varepsilon\right)
+\alpha\left(\int_\Omega f\,v^\varepsilon_+\chi_{_{\Omega_c}}\right)\left(\int_\Omega v^\varepsilon\right)\nonumber\\
&\le&\frac12\left(\int_\Omega v^\varepsilon\right)^2+\alpha^2({m}^\varepsilon_0)^2C+\alpha^2C\|f\|^2\|v^\varepsilon\|^2.\label{F-5}
\end{eqnarray}
Adding (\ref{F-5}) to (\ref{F-2}), then replacing (\ref{F-3}) and (\ref{F-4}) in the resulting inequality, and taking into account that $\alpha\le1$, we obtain
\begin{eqnarray}
&&\frac{d}{dt}\left(\int_\Omega(u^\varepsilon+1)\ln(u^\varepsilon+1)+\frac12\|v^\varepsilon\|^2_{H^1}+\frac\varepsilon2\|\Delta v^\varepsilon\|^2\right)
+2\|\nabla\sqrt{u^\varepsilon+1}\|^2+C\|v^\varepsilon\|^2_{H^2}+\varepsilon\|\nabla(\Delta v^\varepsilon)\|^2\nonumber\\
&&\hspace{0.2cm}\le C(({m}_0^\varepsilon)^2+\|f\|^2_{L^4}\|v^\varepsilon\|^2_{H^1}).\label{F-6}
\end{eqnarray}
From (\ref{F-6}) and Gronwall lemma we deduce that
\begin{eqnarray}
\|v^\varepsilon\|^2_{L^\infty(0,T;H^2(\Omega))}&\le& \frac1\varepsilon\exp(\mathcal{A}(T))\left(\|u_0^\varepsilon\|^2+\|v_0^\varepsilon\|^2_{H^2}+C({m}^\varepsilon_0)^2T\right)\nonumber\\
&:=& K_0^\varepsilon\left({m}_0^\varepsilon,T,\|u_0^\varepsilon\|,\|v_0^\varepsilon\|_{H^2},\mathcal{A}(T)\right),\label{F-7}
\end{eqnarray}
where 
\begin{equation*}\label{F-8}
\mathcal{A}(T):=C\int_0^T\|f(s)\|^2_{L^4}ds=C\|f\|^2_{L^2(L^4)}.
\end{equation*}
Now, integrating (\ref{F-6}) in (0,T) and using (\ref{F-7}) we obtain
\begin{eqnarray}
\int_0^T\|v^\varepsilon(s)\|^2_{H^3}ds&\le&\frac1\varepsilon C\left(\|u^\varepsilon_0\|^2+\|v_0^\varepsilon\|^2_{H^2}+({m}_0^\varepsilon)^2T+(\sup_{0\le s\le T}\|v^\varepsilon(s)\|^2_{H^2})\mathcal{A}(T)\right)\nonumber\\
&:=& K_1^\varepsilon({m}_0^\varepsilon,T,\|u_0^\varepsilon\|,\|v_0^\varepsilon\|_{H^2},\mathcal{A}(T)).\label{F-9}
\end{eqnarray}
Therefore, from (\ref{F-7}) and (\ref{F-9}) we conclude that $v^\varepsilon$ is bounded in $L^\infty(0,T;H^2(\Omega))\cap L^2(0,T;H^3(\Omega))$ (independently of $\alpha\in(0,1]$), which implies that
$z^\varepsilon$ is bounded in $\mathcal{X}.$

\vspace{0.2cm}
\underline{Step 3:} $u^\varepsilon$ is bounded in $\mathcal{X}$.
\vspace{0.2cm}

\noindent Testing (\ref{F-1})$_1$ by $u^\varepsilon$ we have
\begin{equation}\label{F-10}
\frac12\frac{d}{dt}\|u^\varepsilon\|^2+\|\nabla u^\varepsilon\|^2=-\alpha(u^\varepsilon\nabla v^\varepsilon,\nabla u^\varepsilon).
\end{equation}
Applying  H\"older and Young inequalities,  and using (\ref{interpol}), we obtain 
\begin{eqnarray}
-\alpha(u^\varepsilon\nabla v^\varepsilon,\nabla u^\varepsilon)&\le&\alpha\|u^\varepsilon\|_{L^4}\|\nabla v^\varepsilon\|_{L^4}\|\nabla u^\varepsilon\|
\le  C\|u^\varepsilon\|^{1/4}\|\nabla v^\varepsilon\|_{L^4}\|u^\varepsilon\|^{7/4}_{H^1}\nonumber\\
&\le&\frac12\|u^\varepsilon\|^2_{H^1}+C\|\nabla v^\varepsilon\|^8_{L^4}\|u^\varepsilon\|^2. \label{F-11}
\end{eqnarray}
Replacing  (\ref{F-11}) in (\ref{F-10}), and taking into account that $({m}_0^\varepsilon)^2=\left(\displaystyle\int_\Omega u^\varepsilon(t)\right)^2$, we have 
\begin{equation}\label{F-12}
\frac{d}{dt}\|u^\varepsilon\|^2+\|u^\varepsilon\|^2_{H^1}\le C\|\nabla v^\varepsilon\|^8_{L^4}\|u^\varepsilon\|^2+2({m}_0^\varepsilon)^2.
\end{equation}
In particular, using (\ref{interpol}), (\ref{F-7}), we obtain
$$
\|\nabla v^\varepsilon\|^8_{L^4}\le C(K_0^\varepsilon)^4.
$$
Then, we can apply the Gronwall lemma in (\ref{F-12}), obtaining 
\begin{equation}\label{F-13}
\|u^\varepsilon\|^2_{L^\infty(0,T;L^2(\Omega))}\le \exp(C(K_0^\varepsilon)^4)(\|u_0^\varepsilon\|^2+2({m}_0^\varepsilon)^2T):=K_2^\varepsilon({m}_0^\varepsilon,T,\|u_0^\varepsilon\|,\|v_0^\varepsilon\|_{H^2},\mathcal{A}(T)).
\end{equation}
Integrating (\ref{F-12}) in $(0,T)$ we have 
\begin{eqnarray}
\int_0^T\|u^\varepsilon(s)\|^2_{H^1}ds&\le& \|u_0^\varepsilon\|^2+2({m}_0^\varepsilon)^2T+C(K_0^\varepsilon)^4\int_0^T\|u^\varepsilon(s)\|^2ds\nonumber\\
&\le& \|u_0^\varepsilon\|^2+2({m}^\varepsilon_0)^2T+C(K_0^\varepsilon)^4K_2^\varepsilon T\nonumber\\
&:=&K^\varepsilon_3({m}_0^\varepsilon,T,\|u^\varepsilon_0\|,\|v^\varepsilon_0\|_{H^2},\mathcal{A}(T)).\label{F-14}
\end{eqnarray}
Thus, from (\ref{F-13}) and (\ref{F-14}) we deduce that $u^\varepsilon$ is bounded in $\mathcal{X}$. Consequently, the fixed points of $\alpha R$ are bounded 
in $\mathcal{X}\times\mathcal{X}$, independently of $\alpha>0$. For $\alpha=0$ the result
is trivial.
\end{proof}

\begin{lemma}\label{conti}
The operator  $R:\mathcal{X}\times\mathcal{X}\rightarrow\mathcal{X}\times\mathcal{X}$,
defined in (\ref{regg-2}), is continuous.
\end{lemma}
\begin{proof}
Let $\{(\overline{u}^\varepsilon_m,\overline{z}^\varepsilon_m)\}_{m\in\mathbb{N}}\subset \mathcal{X}\times\mathcal{X}$ be a sequence such that
\begin{equation}\label{cont-1}
(\overline{u}^\varepsilon_m,\overline{z}^\varepsilon_m)\rightarrow(\overline{u}^\varepsilon,\overline{z}^\varepsilon)\mbox{ in }\mathcal{X}\times\mathcal{X}.
\end{equation}
In particular, $\{(\overline{u}^\varepsilon_m,\overline{z}^\varepsilon_m)\}_{m\in\mathbb{N}}$ is bounded in $\mathcal{X}\times\mathcal{X}$, thus, from (\ref{com-1}) and (\ref{com-2})
we deduce that sequence $\{(u^\varepsilon_m,z^\varepsilon_m):=R(\overline{u}^\varepsilon_m,\overline{z}^\varepsilon_m)\}_{m\in\mathbb{N}}$ is bounded in $X_{5/3}\times X_{10/3}$. Then, there exists 
a subsequence of $\{R(\overline{u}^\varepsilon_m,\overline{z}^\varepsilon_m)\}_{m\in\mathbb{N}}$, still denoted by $\{R(\overline{u}^\varepsilon_m,\overline{z}^\varepsilon_m)\}_{m\in\mathbb{N}}$, and an
element $(\widehat{u}^\varepsilon,\widehat{z}^\varepsilon)\in X_{5/3}\times X_{10/3}$ such that
\begin{equation}\label{cont-2}
R(\overline{u}^\varepsilon_m,\overline{z}^\varepsilon_m)\rightarrow (\widehat{u}^\varepsilon,\widehat{z}^\varepsilon)\mbox{ weakly in }X_{5/3}\times X_{10/3}
\mbox{ and strongly in }\mathcal{X}\times\mathcal{X}.
\end{equation}
Now, we consider system (\ref{regg-2}) written for $(u^\varepsilon,z^\varepsilon)=R(\overline{u}^\varepsilon_m,\overline{z}^\varepsilon_m)$ and
$(\overline{u}^\varepsilon,\overline{z}^\varepsilon)=(\overline{u}^\varepsilon_m,\overline{z}^\varepsilon_m)$. From (\ref{cont-1}) and (\ref{cont-2}), taking the limit in the system
depending on $m$, as $m$ goes to $+\infty$, we deduce that $(\widehat{u}^\varepsilon,\widehat{z}^\varepsilon)=R(\lim_{m\rightarrow+\infty}(\overline{u}^\varepsilon_m,\overline{z}^\varepsilon_m))$.
Then, by uniqueness of limit the whole sequence 
$\{R(\overline{u}^\varepsilon_m,\overline{z}^\varepsilon_m)\}_{m\in\mathbb{N}}$ converges to $(\widehat{u}^\varepsilon,\widehat{z}^\varepsilon)$ strongly in $\mathcal{X}\times\mathcal{X}$.
Thus, operator $R:\mathcal{X}\times\mathcal{X}\rightarrow\mathcal{X}\times\mathcal{X}$ is continuous. 
\end{proof}

Consequently, from Lemmas \ref{compact}, \ref{fix} and \ref{conti}, it follows that the operator $R$ satisfy the hypotheses of the Leray-Schauder fixed point theorem. Thus, we conclude that the
map $R$ has a fixed point  $(u^\varepsilon,z^\varepsilon)$, that is $R(u^\varepsilon,z^\varepsilon)=(u^\varepsilon,z^\varepsilon)$, which is a solution of 
system (\ref{regg-1}).

\subsection*{Acknowledgments}

F. Guill\'en-Gonz\'alez and M.A. Rodr\'iguez-Bellido have been supported by MINECO grant MTM2015-69875-P (Ministerio de Econom\'ia y Competitividad, Spain) with the participation of
FEDER. E. Mallea-Zepeda has been supported by Proyecto UTA-Mayor 4740-18 (Universidad de Tarapac\'a, Chile). Also, E. Mallea-Zepeda expresses his gratitude to Instituto de Matem\'aticas 
Universidad de Sevilla and Dpto. de Ecuaciones Diferenciales y An\'alisis Num\'erico of Universidad de Sevilla for their hospitality during his research stay in both centers.

\end{document}